\documentclass[12pt]{article} 
\usepackage[sectionbib]{natbib}
\usepackage{array,epsfig,fancyheadings,rotating}
\usepackage[]{hyperref}  
\usepackage{sectsty, secdot}
\usepackage{algorithm}
\usepackage{algorithmicx}
\usepackage{setspace}
\sectionfont{\fontsize{12}{14pt plus.8pt minus .6pt}\selectfont}
\renewcommand{\theequation}{\thesection\arabic{equation}}
\subsectionfont{\fontsize{12}{14pt plus.8pt minus .6pt}\selectfont}

\usepackage{amsmath}
\usepackage{amssymb}
\usepackage{amsfonts}
\usepackage{multirow}
\usepackage{amsthm}

\newcommand\norm[1]{\left\lVert#1\right\rVert}
\newtheorem{theorem}{Theorem}
\newtheorem{lemma}{Lemma}

\theoremstyle{definition}

\begin{document}


\renewcommand{\baselinestretch}{2}

\markright{ \hbox{\footnotesize\rm Statistica Sinica
}\hfill\\[-13pt]
\hbox{\footnotesize\rm
}\hfill }

\markboth{\hfill{\footnotesize\rm FIRSTNAME1 LASTNAME1 AND FIRSTNAME2 LASTNAME2} \hfill}
{\hfill {\footnotesize\rm Selective Confidence Intervals} \hfill}

\renewcommand{\thefootnote}{}
$\ $\par


\fontsize{12}{14pt plus.8pt minus .6pt}\selectfont \vspace{0.8pc}
\centerline{\large\bf Selective Confidence Intervals for Martingale}
\vspace{2pt} \centerline{\large\bf  Regression Model}
\vspace{.4cm} \centerline{Ka Wai Tsang and Wei Dai} \vspace{.4cm} \centerline{\it
The Chinese University of Hong Kong, Shenzhen} \vspace{.55cm} \fontsize{9}{11.5pt plus.8pt minus
.6pt}\selectfont


\begin{quotation}
\noindent {\it Abstract:}
In this paper we consider the problem of constructing confidence intervals for coefficients of martingale regression models (in particular, time series models) after variable selection. Although constructing confidence intervals are common practice in statistical analysis, it is challenging in our framework due to the data-dependence of the selected model and the correlation among the variables being selected and not selected. We first introduce estimators for the selected coefficients and show that it is consistent under martingale regression model, in which the observations can be dependent and the errors can be heteroskedastic. Then we use the estimators together with a resampling approach to construct confidence intervals. Our simulation results show that our approach outperforms other existing approaches in various data structures.

\vspace{9pt}
\par
\end{quotation}\par

\def\thefigure{\arabic{figure}}
\def\thetable{\arabic{table}}

\renewcommand{\theequation}{\thesection.\arabic{equation}}

\fontsize{12}{14pt plus.8pt minus .6pt}\selectfont

\setcounter{section}{1} 
\setcounter{equation}{0} 

\lhead[\footnotesize\thepage\fancyplain{}\leftmark]{}\rhead[]{\fancyplain{}\rightmark\footnotesize\thepage}

\noindent {\bf 1. Introduction}

Consider the linear regression  model
\begin{align}
y_t=\sum_{j=1}^{p}\beta_j x_{tj}+\varepsilon_t,\quad t=1,\ldots,n, \label{Linear Model}
\end{align}
with $p$ predictor variables $\mathbf{x}_t=(x_{t1},\ldots,x_{tp})^T$ and $n$ samples that can be correlated. The error terms $\varepsilon_t$ are usually assumed to be independent of $\mathbf{x}_t$ for many applications. However, it is too strong for the regression models in financial time series. Instead, this paper considers the martingale regression which has the form \eqref{Linear Model} with $\{\varepsilon_t\}$ being a local martingale sequence and the components of $\mathbf{x}_t$ containing lagged variables $y_{t-1},y_{t-2},\ldots$ and other factor variables. The well-known AR(p)-GARCH(h,k) model is a special case with $\mathbf{x}_t=(y_{t-1},\ldots,y_{t-p})^T$ and 
\begin{align}
\varepsilon_t=\sigma_t\xi_t,\quad \sigma_t^2=\omega+\sum_{i=1}^{h}b_i\sigma_{t-i}^2+\sum_{j=1}^{k}a_j\varepsilon^2_{t-j},\label{Garch Model}
\end{align}
in which $\xi_t$ are i.i.d.\ with mean 0 and variance 1; see Section 2.6.3 of Guo et al.\ (2017). We are interested in estimating confidence intervals for some selected unknown coefficients $\beta_j$ in \eqref{Linear Model} when $p>n$. Zhang and Zhang (2014) consider a similar problem for deterministic $\mathbf{X}$ and $\varepsilon_t$ are independent for coefficients selected by scaled lasso. Belloni et al.\ (2015) develop uniformly valid confidence regions for coefficients selected by some lasso-type methods under the assumptions that $(\mathbf{x}_t,\varepsilon_t)$ are i.i.d.\ and $\varepsilon_t$ is independent of $\mathbf{x}_t$. Lee and Wu (2018) propose a bootstrap method to estimate the distributions of the least squares estimators after some data-driven model selection procedure. They consider independent $\mathbf{x}_t$ and $\varepsilon_t$ and assume that $p$ is constant. However, the problem of confidence intervals estimation after model selection for the martingale regression model \eqref{Linear Model} is still largely untouched. Let $\mathbf{X}=(\mathbf{x}_1,\ldots,\mathbf{x}_n)^T$, $\mathbf{Y}=(y_1,\ldots,y_n)^T$, $\boldsymbol{\varepsilon}=(\varepsilon_1,\ldots,\varepsilon_n)^T$, and $\boldsymbol{\beta}=(\beta_1,\ldots,\beta_p)^T$, then model \eqref{Linear Model} can be written as 
\begin{align}
\mathbf{Y}=\mathbf{X}\boldsymbol{\beta}+\boldsymbol{\varepsilon}. \label{Linear Model in Capital}
\end{align}
When $\mathbf{X}$ is a nonrandom full rank matrix with $n>p$, and $\varepsilon_t$ are independent $\mathcal{N}(0,\sigma^2)$, it is well-known that 
\begin{align}
\frac{\hat{\beta}^{\text{ols}}_j-
	\beta_j}{s\sqrt{c_{jj}}}\sim t_{n-p},\quad j=1,\ldots,p, \label{t test statistic}
\end{align}
where $c_{jj}$ is the $j$th diagonal element of $(\mathbf{X}^T\mathbf{X})^{-1}$, $\hat{\boldsymbol{\beta}}^{\text{ols}}=(\hat{\beta}_1^{\text{ols}},\ldots,\hat{\beta}_p^{\text{ols}})=(\mathbf{X}^T\mathbf{X})^{-1}\mathbf{X}^T\mathbf{Y}$ is the ordinary least squares (OLS) estimate, and $s^2=\sum_{t=1}^{n}(y_t-\mathbf{x}_t^T\hat{\boldsymbol{\beta}}^{\text{ols}})^2/(n-p)$ is the sample variance of  $\varepsilon_t$. If we assume that $\mathbf{X}^T\mathbf{X}/c_n$ converges in probability to a nonrandom matrix with positive eigenvalues for some nonrandom constants $c_n$ such that $\lim_{n\rightarrow \infty}c_n=\infty$, then the results in \eqref{t test statistic} still holds asymptotically under some additional regularity conditions, see Section 1.5.3 of Lai and Xing (2008). However, when $p>n$, the convergence of $\mathbf{X}^T\mathbf{X}/c_n$ mentioned above becomes infeasible and estimating confidence intervals for $\beta_j$ becomes challenging. Actually as the OLS estimator cannot be applied due to the singularity of $\mathbf{X}^T\mathbf{X}$, coefficients estimation for high-dimensional ($p>n$) regression model has been a long standing problem in statistics.

With the assumption that $\boldsymbol{\beta}$ satisfies certain sparsity conditions, the issues due to high dimension can be partly solved by selecting a subset $J \subset \{1,\ldots,p\}$ of size $m\ll n$ and assuming $\beta_j=0$ for $j\notin J$ to reduce the dimension. The traditional idea of best subset selection methods are first introduced by Efroymson (1960), and lead to two popular criteria for model selection, Akaike's AIC and Schwarz's BIC. AIC in Akaike (1973, 1974) chooses a model that minimizes the Kullback-Lerbler (KL) divergence of the fitted model from the true model, while BIC in Schwarz (1978) chooses a model that minimizes a criterion that is formed by a Bayesian approach. Other criteria with further development have been proposed by Hannan and Quinn (1979), Rao and Wu (1989), Wei (1992), Hurvich and Tsai (1991) and Shao (1997). Since it is infeasible to try all possible models to find the the one with minimum criterion's value in high-dimensional regression, they are usually carried out in forward stepwise manner. Ing and Lai (2011) proposes a high-dimensional information criterion (HDIC) to choose the best model along a path of models that are selected by a forward stepwise method called orthogonal greedy algorithm (OGA). HDIC is similar to AIC and BIC with the penalty term in AIC or BIC multiplied by $\log p$ (namely HDAIC and HDBIC correspondingly in Ing and Lai (2011)) for the case $p\gg n$. Other forward stepwise algorithms for high-dimensional regression can be found in B{\"u}hlmann (2006), Chen and Chen (2008), Wang (2009) and Fan and Lv (2008). Another popular approach for variable selection is by penalized least squares estimators. Tibshirani (1996) proposes an estimator that minimize an objective function consisting of the least squares errors and a $\text{L}_1$ penalty term $\norm{\boldsymbol{\beta}}_1=\sum_{i=1}^{p}|\beta_j|$. Subsequent modifications and refinements of penalized least squares methods are developed by Zou and Hastie (2005), Zou (2006), Yuan and Lin (2006), Bickel et al.\ (2009), and Zhang (2010). Although many variable selection methods have been proposed, as Ing (2019) points out that ``the vast majority of studies on model \eqref{Linear Model}, however, have focused on situations where $\mathbf{x}_t$ are nonrandom and $\varepsilon_t$ are independently and identically distributed (i.i.d.) or $(\mathbf{x}_t,\varepsilon_t)$ are i.i.d., which regrettably preclude most serially correlated data." Ing (2019) gives analysis of OGA for high-dimensional regression models with dependent observations $(\mathbf{x}_t,y_t)$ and shows the convergence of the prediction error of OGA. Therefore, we also choose OGA, which is presented in Section 3.1, to do variable selection on model \eqref{Linear Model}. However, note that our theoretical results in this paper do not require any properties of OGA, and thus also hold for other selection methods for martingale regression model \eqref{Linear Model}. Let $\hat{J}$ of size $m$ be the selected set after $m$ OGA iterations, then model \eqref{Linear Model} can be written as 
\begin{align}
y_t=\sum_{j\in\hat{J}}\beta_jx_{tj}+\sum_{j\in\hat{J}^c}\beta_jx_{tj}+\varepsilon_t,\quad t=1,\ldots,n, \label{Linear Model decomposition}
\end{align}
and we are interested in the confidence intervals of $\beta_j,j\in \hat{J}$. Denote $\mathbf{M}_J$ as the sub-matrix of $\mathbf{M}=(\mathbf{M}_1,\ldots,\mathbf{M}_p)$ such that $\mathbf{M}_J=(\mathbf{M}_j)_{j\in J}$  and $\mathbf{v}_J=(\mathbf{v}_j)_{j \in J}$ as the sub-vector of $\mathbf{v}=(v_1,\ldots,v_p)^T$. The matrix-vector form of model \eqref{Linear Model decomposition} is 
\begin{align}
\mathbf{Y}=\mathbf{X}_{\hat{J}}\boldsymbol{\beta}_{\hat{J}}+\mathbf{w}, \label{Linear Model decomposition in Capital}
\end{align}
where $\mathbf{w}=\mathbf{X}_{\hat{J}^c}\boldsymbol{\beta}_{\hat{J}^c}+\boldsymbol{\varepsilon}$. If we replace $x_{tj}$ and $y_t$ in \eqref{Linear Model decomposition} by $x_{tj}-\mu_{x,j}$ and $y_t-\mu_y$, where $\mu_{x,j}$ and $\mu_y$ are the unconditional expectations of the weakly stationary time series $x_{tj}$ and $y_t$, then the equality in \eqref{Linear Model decomposition} still holds. Hence, without loss of generality we assume $E\mathbf{x}_t=\boldsymbol{0}$ and thus $E\mathbf{w}=\boldsymbol{0}$, which makes $\mathbf{w}$ look like an error term. It is a common practice that the confidence intervals for $\boldsymbol\beta_{\hat{J}}$ in \eqref{Linear Model decomposition in Capital} are constructed by assuming \eqref{t test statistic} hold with $\mathbf{X}$ being replaced by $\mathbf{X}_{\hat{J}}$. However, such confidence intervals are invalid based on two major effects: {\it selection effect} and {\it spill-over effect}.

\paragraph{Selection effect:}
Usually people focus on the coefficients in the selected set after model selection. However, as noted by Sori{\'c} (1989) that ``in a large number of 95\% confidence intervals, 95\% of them contain the population parameter (e.g., difference between population means); but it would be wrong to imagine that the same rule also applies to a large number of 95\% NZ (not containing zero) confidence intervals." To illustrate this statement, suppose there are $X_i\sim N(\mu_i,1)$ for $i=1,\ldots,5$. The confidence interval for each $\mu_i$ can be constructed by the normal distribution, e.g. $X_i\pm z_{1-\alpha/2}$ for a $100(1-\alpha)\%$ confidence interval, where $ z_{1-\alpha/2}$ is the $(1-\alpha/2)^{\text{th}}$ quantile of $N(0,1)$. However, if we are interested in the confidence interval of $X_{\text{I}_{\max}}$, where $\text{I}_{\max}=\arg \max_i X_i$, then $X_{\text{I}_{\max}}$ is no longer $N(\mu_{\text{I}_{\max}},1)$ distributed. In particular, we can see that if $X_i$ are i.i.d.\ $N(0,1)$, $EX_{\text{I}_{\max}}> \mu_{\text{I}_{\max}}=0$. Therefore , $X_{\text{I}_{\max}}\pm z_{1-\alpha/2}$ is not a valid confidence interval for $\mu_{\text{I}_{\max}}$. Similarly, many model selection methods select the best $m$ out of $p$ variables based on different selection criteria, and hence the distribution of $\hat{\beta}_j^{\text{ols}}, j\in \hat{J}$, may not approximately follow the distribution in \eqref{t test statistic}. We call this effect, which causes $E\hat{\beta}_j^{\text{ols}}\neq \beta_j$ due to model selection, as selection effect.

\paragraph{Spill-over effect:}
With the assumption that $x_{tj}$ are nonrandom and $\varepsilon_t$ are normal distributed in model \eqref{Linear Model}, Taylor et al.\ (2014) have developed conditional distributions for entries in $E\hat{\boldsymbol{\beta}}_{\hat{J}}^{\text{ols}}$ given that $\hat{J}$ is selected by certain types of methods, including OGA; see Lee and Taylor (2014, Section 8.2). The conditional distributions can then be used to construct valid confidence intervals for the entries of $E\hat{\boldsymbol{\beta}}_{\hat{J}}^{\text{ols}}$. However, since
\begin{align}
E\hat{\boldsymbol{\beta}}_{\hat{J}}^{\text{ols}}=\boldsymbol{\beta}_{\hat{J}}+\sum_{j\in \hat{J}^c}\beta_j(\mathbf{X}_{\hat{J}}^T\mathbf{X}_{\hat{J}})^{-1}\mathbf{X}_{\hat{J}}^T\mathbf{X}_j, \label{Biase of estimator}
\end{align}
when $\mathbf{X}$ is nonrandom. Therefore, unless $\beta_j=0$ (all relevant variables are selected) or $\|\mathbf{X}_{\hat{J}}^T\mathbf{X}_j\|=0$ (orthogonal) for $j\notin \hat{J}$, otherwise $E\hat{\boldsymbol{\beta}}_{\hat{J}}^{\text{ols}}\neq \boldsymbol{\beta}_{\hat{J}}$. Ing et al.\ (2017) have noticed this problem and called it spill-over effect. One way to avoid spill-over effect is assuming all relevant predictors $(\beta_j\neq 0)$ are selected asymptotically. Based on this assumption, Belloni et al.\ (2014) and Voorman et al.\ (2014) constructed an asymptotically normal estimator for coefficient $\beta_j$ and thus can construct asymptotically valid $p$-values. Lockhart et al. (2014) propose a test statistic for each newly selected variable. The distribution of the statistic is asymptotically $Exp(1)$ under the null hypothesis that all relevant predictors have been selected before the newly entered variable. 

In this paper, we handle selection effect and spill-over effect by a consistent estimator of $\boldsymbol\beta_{\hat{J}}$ and a resampling approach. Suppose we know the true $\boldsymbol{\beta}_{\hat{J}}$ and the distribution $\mathbf{w}$ in \eqref{Linear Model decomposition in Capital} such that we can generate 
$\mathbf{w}^{(b)}$ and $\mathbf{Y}^{(b)}=\mathbf{X}_{\hat{J}}\boldsymbol{\beta}_{\hat{J}}+\mathbf{w}^{(b)}$ with the same distribution as $\mathbf{w}$ and $\mathbf{Y}$. Then, to determine if $\theta$ belongs to a confidence interval for $\beta_j, j\in \hat{J}$, we can construct a test statistic $T_j=T_j(\mathbf{X},\mathbf{Y},\theta)$, e.g.\ the test statistic in \eqref{t test statistic} with $\mathbf{X}$ being replaced by $\mathbf{X}_{\hat{J}}$, and compare it with simulated $T_j^{(b)}=T_j(\mathbf{X},\mathbf{Y}^{(b)}_j,\theta)$, $b=1,\ldots, B$, where $\mathbf{Y}^{(b)}_j=\sum_{i \in \hat{J}\backslash\{j\}}\beta_i \mathbf{X}_i +\theta \mathbf{X}_j+ \mathbf{w}^{(b)}$. We exclude $\theta$ in a confidence interval for $\beta_j$ if $T_j$ is an extreme value to the empirical distribution formed by $T_j^{(b)}$. If $\beta_j=\theta$, then $T_j^{(b)}$ has the same distribution as $T_j$ and hence their difference should not be significant, and $\beta_j$ should not be excluded in the confidence interval. This is the idea of exact method introduced in Chuang and Lai (2000). As the model parameters are unknown in practice, Chuang and Lai (2000) ``hybridize" the exact method and bootstrap resampling to develop a resampling method called hybrid resampling for constructing confidence intervals. We follow their approach to estimate valid confidence for $\beta_j$. We first present our estimators for selected coefficients and generating mechanism of $\mathbf{w}^{(b)}$ in Section 2. Assumptions and theorems for the consistency of our estimators are presented there. Section 3.1 gives an introduction of OGA, which is used for variable selection for the entire paper. Our test statistic functions $T_j(\mathbf{X},\mathbf{Y},\theta)$ for $j \in \hat{J}$ are described in Section 3.2. Collecting the results in Section 2 and 3, Section 3.3 presents a hybrid resampling approach to construct confidence intervals for $\beta_j, j \in \hat{J}$. Simulation studies to illustrate our theoretical results and the performance of our algorithms are presented in Section 4. Section 5 gives further discussion and some concluding remarks.

\section{$\boldsymbol{\beta}_{\hat{J}}$ estimation and $\mathbf{w}^{(b)}$ generation }
We first consider a subset $J$ of size $m$ that is not selected based on observed $\mathbf{X}$ and $\mathbf{Y}$. In such case, there is no selection effect but may still have spill-over effect if the columns of $\mathbf{X}$ are correlated and $|\beta_j|>0$ for some $ j\notin J$. The model we consider in Section 2.1 is 
\begin{align}
\mathbf{Y}=\mathbf{X}_J\boldsymbol{\beta}_J+\boldsymbol{w}_J, \label{Decomposition of Y}
\end{align}
where $\boldsymbol{w}_J=(w_{J1},\ldots,w_{Jn})^T=\mathbf{X}_{J^c}\boldsymbol{\beta}_{J^c}+\boldsymbol{\varepsilon}$.

\subsection{Consistent $\boldsymbol{\beta}_{J}$ estimation}
From \eqref{Biase of estimator}, we see that the OLS estimator $\hat{\boldsymbol{\beta}}_J^{\text{ols}}=(\mathbf{X}_J^T\mathbf{X}_J)^{-1}\mathbf{X}_J^T\mathbf{Y}$ is not a consistent estimator for $\boldsymbol{\beta}_J$ in general. However, if there exists $\mathbf{Z}=(\mathbf{Z}_1,\ldots,\mathbf{Z}_n)^T$ such that (i) $\mathbf{Z}_t$ is uncorrelated with $w_{Jt}$ and (ii) $\frac{1}{n}\mathbf{Z}^T\mathbf{X}_J$ converges in probability to a nonsingular matrix, then the method of instrumental variables described in Section 9.7.1 of Lai and Xing (2008) can be applied to give a consistent estimator $\hat{\boldsymbol{\beta}}_J^{\text{IV}}=(\hat{\mathbf{X}}^T_J\hat{\mathbf{X}}_J)^{-1}\hat{\mathbf{X}}^T_J\mathbf{Y}$, where $\hat{\mathbf{X}}_J=\mathbf{Z}(\mathbf{Z}^T\mathbf{Z})^{-1}\mathbf{Z}^T\mathbf{X}_J$. To apply a similar idea of instrumental variables, we made the following assumptions on model \eqref{Linear Model}.\\

\noindent
\textbf{Assumption A} (Model Design)
\begin{itemize}
	\item [A1.]$x_{tj}=\lambda_j^Tf_t+e_{tj}$, where $f_t\in\mathbb{R}^r$ is a vector of common factors, $\lambda_j$ is a vector of factor loadings associated with $f_t$, and $e_{tj}$ is the idiosyncratic component of $x_{tj}$.
	\item [A2.] $p\geq O_p(n)$, $\lim_{n\rightarrow \infty}\frac{m^2}{n}=0$
	\item [A3.] $\sum_{j=1}^{p}|\beta_j|\leq M $ for some constant $M$.
\end{itemize}

The factor model for $x_{tj}$ in Assumption A1 is considered by Bai and Ng (2002). They suggest using factor model for analyzing financial data and point out that ``the idea that variations in a large number of economic variables can be modeled by a small number of reference variables is appealing and is used in many economic analyses. For example, asset returns are often modeled as a function of a small number of factors$\ldots$ Stock and Watson (1989) showed that the forecast error of a large number of macroeconomic variables can be reduced by including diffusion indexes, or factors, in structural as well as nonstructural forecasting models. In demand analysis, Engel curves can be expressed in terms of a finite number of factors. Factor analysis also provides a convenient way to study the aggregate implications of microeconomic behavior, as shown in Forni and Lippi (1997)." Assumption A2 allows $p\gg n$, but the number $m$ of selected variables cannot increase too fast as $n\rightarrow \infty$. Ing and Lai (2011) suggest $m=O(\sqrt{n/\log p})$,  which satisfies Assumption A2, to be the number of selection for OGA. Assumption A3 allows the number of relevant variables $(|\beta_j|>0)$ greater than $m$ and the sum of the absolute values of $\beta_j$ outside the selection set, i.e. $\sum_{j\in J^c}|\beta_j|$, does not converge to 0 as $n\rightarrow \infty$.

Let $\mathbf{F}=(f_1,\ldots,f_n)^T\in \mathbb{R}^{n\times r}$, $\mathbf{\Lambda}=(\lambda_1,\ldots,\lambda_p)^T\in \mathbb{R}^{p\times r}$, $\mathbf{E}_j=(e_{1j},\ldots,e_{nj})^T$, and $\mathbf{E}=(\mathbf{E}_1,\ldots,\mathbf{E}_p)$. Then we have 
\begin{align}
\mathbf{X}_j=\mathbf{F}\lambda_j+\mathbf{E}_j,\quad  \mathbf{X}=\mathbf{F}\mathbf{\Lambda}^T+\mathbf{E} \label{Common Factor}
\end{align}
and the model \eqref{Decomposition of Y} can be written as 
\begin{align}
\mathbf{Y}=(\mathbf{I}-\mathbf{P}_F)\mathbf{X}_J\boldsymbol{\beta}_J+\mathbf{P}_F\mathbf{X}_J\boldsymbol{\beta}_J+\mathbf{w}_J=\tilde{\mathbf{X}}_J\boldsymbol{\beta}_J+\tilde{\mathbf{w}}_J, \label{Decompositon of Y_2}
\end{align}
where $\mathbf{P}_F=\mathbf{F}(\mathbf{F}^T\mathbf{F})^{-1}\mathbf{F}$, $\tilde{\mathbf{X}}_J=(\mathbf{I}-\mathbf{P}_F)\mathbf{X}_J$, $\tilde{\mathbf{w}}_{J}=\mathbf{P}_F\mathbf{X}_J\boldsymbol{\beta}_J+\mathbf{w}_{J}$. If we set $\mathbf{Z}=\tilde{\mathbf{X}}_J$, then we can check that $E(\mathbf{Z}^T\tilde{\mathbf{w}}_{J})=\mathbf{0}$ if $f_t$, $e_{tj}$ and $\varepsilon_t$ are all independent and $E(e_{tj})=0$. If we further have $\frac{1}{n}\mathbf{Z}^T\tilde{\mathbf{X}}_J=\frac{1}{n}\tilde{\mathbf{X}}_J^T\tilde{\mathbf{X}}_J$ converges to a nonsingular matrix in probability, which is stated in Theorem 3, then we can estimate $\boldsymbol{\beta}_J$ consistently by the method of instrumental variables. However, the factor matrix $\mathbf{F}$ and its rank $r=rank(\mathbf{F})$ are unknown in practice. Bai and Ng (2002) propose the following procedure to estimate $r$ and $F$.
\begin{algorithm}[h]
	\begin{algorithmic}[]
		\State \textbf{INPUT}: $\mathbf{X}\in \mathbb{R}^{n\times p}$,$\mathbf{Y}\in \mathbb{R}^{n}$
		\State \textbf{Step 1}:  For $k=1$ to $k_{\max}$\\
		\begin{enumerate}
			\item [1.1] Calculate factor loadings $\bar{\mathbf{\Lambda}}^k$, which is constructed as $\sqrt{p}$ times the eigenvectors corresponding to the $k$ largest eigenvalues of the $p\times p$ matrix $\mathbf{X}\mathbf{X}^T$. 
			\item [1.2] Calculate $\bar{\mathbf{F}}^k=\mathbf{X}\bar{\mathbf{\Lambda}}^k/p$ and  rescaled factors $\hat{\mathbf{F}}^k=\bar{\mathbf{F}}^k((\bar{\mathbf{F}}^k)^T\bar{\mathbf{F}}^k/n)^{1/2}$.
			\item [1.3] Compute $V(k)=\min_{\mathbf{\Lambda}^k}\norm{\mathbf{X}-\hat{\mathbf{F}}^k(\mathbf{\Lambda}^k)^T}^2$, where $\norm{\cdot}$ is the Frobenius norm ($\norm{A}^2=\sum_{i}\sum_{j}A_{ij}^2$).
		\end{enumerate}
		\State \textbf{Step 2}: Choose $\hat{k}$ that minimize 
		\begin{align*}
		\text{IC}(k)=\log(V(k))+k\Big(\frac{n+p}{np}\Big)\log \Big(\frac{np}{n+p}\Big)
		\end{align*}
		\State  \textbf{OUTPUT}: $\hat{k}$ and $\hat{\mathbf{F}}^{\hat{k}}$
	\end{algorithmic}
	\caption{ Factors and rank estimation}
	\label{alg:algorithm1}
\end{algorithm}

Bai and Ng (2002) show the convergence of the estimated rank and factor matrix under the following assumptions.\\

\noindent
\textbf{Assumption B} (Factors)
\begin{itemize}
	\item[B1.]$E\norm{f_t}^4<\infty$ and $\frac{1}{n}\mathbf{F}^T\mathbf{F}\rightarrow \mathbf{\Sigma}_F$ as $n\rightarrow \infty$ for some positive definite matrix $\mathbf{\Sigma}_F \in \mathbb{R}^{r\times r}$.
	\item[B2.]  $\norm{\lambda_i}\leq \bar{\lambda} < \infty$ and $\norm{\frac{1}{p}\mathbf{\Lambda}^T\mathbf{\Lambda}-\mathbf{\Sigma}_\Lambda}\rightarrow 0$ as $p\rightarrow \infty$ for some  positive definite matrix $\mathbf{\Sigma}_\Lambda \in \mathbb{R}^{r\times r}$.
	\item[B3.] For the same constant $M$ in Assumption A3, assume
	\begin{enumerate}
		\item $E(e_{ti})=0$, $E|e_{ti}|^8\leq M$;
		\item $E(e_s^Te_t/p)=E(p^{-1}\sum_{i=1}^{p}e_{si}e_{ti})=\gamma_p(s,t),|\gamma_p(s,s)|\leq M$ for all $s$, and $n^{-1}\sum_{s=1}^{n}\sum_{t=1}^{n}|\gamma_p(s,t)|\leq M$.
		\item $E(e_{ti}e_{tj})=\tau_{ij,t}$ with $|\tau_{ij,t}|\leq |\tau_{ij}|$  for some $\tau_{ij}$ and for all $t$; in addition, $N^{-1}\sum_{i=1}^{p}\sum_{j=1}^{p}|\tau_{ij}|\leq M$
		\item $E(e_{ti}e_{sj})=\tau_{ij,ts}$ and $(np)^{-1}\sum_{i=1}^{p}\sum_{j=1}^{p}\sum_{t=1}^{n}\sum_{s=1}^{n}|\tau_{ij,ts}|\leq M$
		\item for every $(t,s)$, $E|p^{-1/2}\sum_{i=1}^{p}[e_{si}e_{ti}-E(e_{si}e_{ti})]|^4\leq M$
	\end{enumerate}
	\item[B4.] Weak dependence between $\mathbf{F}$ and $\mathbf{E}$
	\begin{align*}
	E\left(\frac{1}{p}\sum_{i=1}^{p}\norm{\frac{1}{\sqrt{n}}\sum_{t=1}^{n}f_te_{ti}}^2 \right)\leq M
	\end{align*}
\end{itemize}
We present the following 2 theorems from Bai and Ng (2002) as our Theorems 1 and 2.
\begin{theorem}
	Under Assumptions B1 to B4, for any fixed $k\geq 1$, there exists a $(r\times k)$ matrix $\mathbf{H}^k$ with $\text{rank}(\mathbf{H}^k)=\min(k,r)$, such that
	\begin{align}
	\Big(\frac{1}{n}\sum_{t=1}^{n}\norm{\hat{f}_t^k-(\mathbf{H}^k)^Tf_t}^2\Big)=O_p(
	\frac{1}{n})
	\end{align}
	where $\hat{f}_t^k$ are the rows of $\hat{\mathbf{F}}^k$  in Algorithm 1
\end{theorem}

\begin{theorem}
	Under the Assumptions B1 to B4, $\lim_{n\rightarrow \infty}\mathbf{P}(\hat{k}=r)=1$.
\end{theorem}

\begin{algorithm}[h]
	\begin{algorithmic}[]
		
		\State \textbf{INPUT}: $\mathbf{X}\in \mathbb{R}^{n\times p}$, $\mathbf{Y} \in \mathbb{R}^n$, $J \subset \{1,\ldots,p\}$
		\State \textbf{Step 1}: Compute $\hat{\mathbf{F}}=\hat{\mathbf{F}}^{
			\hat{k}}$ by Algorithm 1 based on $\mathbf{X},\mathbf{Y}$.
		
		\State \textbf{Step 2}: Compute $\mathbf{\tilde{X}}_J=(\mathbf{I}-\hat{\mathbf{F}}(\hat{\mathbf{F}}^T\hat{\mathbf{F}})^{-1}\hat{\mathbf{F}}^T)\mathbf{X}_J$
		\State \textbf{Step 3}: $\tilde{\boldsymbol{\beta}}_J^0=(\tilde{\mathbf{X}}_J^T\tilde{\mathbf{X}}_J)^{-1}\tilde{\mathbf{X}}_J^T\mathbf{Y}$
		
		\State \textbf{OUTPUT}: $\tilde{\boldsymbol{\beta}}_J^0$.
	\end{algorithmic}
	\caption{ $\boldsymbol{\beta}_J$ estimation}
	\label{alg:algorithm 2}
\end{algorithm}

We present our algorithm to estimate $\boldsymbol{\beta}_J$ in \eqref{Decomposition of Y} in Algorithm 2. To show the consistency of $\tilde{\boldsymbol{\beta}}_J^0$, we also make the following assumptions on the selection set $J$.\\

\noindent
\textbf{Assumption C} (Selection set)
\begin{itemize}
	\item[C1.] $\frac{1}{n}\mathbf{E}_J^T\mathbf{E}_J\rightarrow \mathbf{G}_J$ as $n\rightarrow \infty$ for some positive definite matrix $\mathbf{G}_J$.
	\item[C2.] $E(\frac{1}{m}\sum_{j\in J}\norm{\frac{1}{\sqrt{n}}\sum_{t=1}^{n}f_t e_{tj}}^2)\leq M$
	\item[C3.] $E(\frac{1}{m}\sum_{j\in J}\frac{1}{\sqrt{n}}\sum_{t=1}^{n}\varepsilon_t e_{tj}^2)\leq M$
	\item[C4.] For $i \notin J$, $E(\frac{1}{m}\sum_{j\in J}(\frac{1}{\sqrt{n}}\sum_{t=1}^{n}e_{ti} e_{tj})^2)\leq M$
\end{itemize}
Assumption C1 is required for the inverse $(\tilde{\mathbf{X}}_J^T\tilde{\mathbf{X}}_J)^{-1}$ in Algorithm 2 to be reasonable. Assumption C2 to C4 are similar to B4 for the weak dependence between $f_t,\varepsilon_t,e_{ti}$ with $e_{tj}$ for $i\notin J$ and $j \in J$. Examples 1 and 2 justify our assumptions for martingale regression model \eqref{Linear Model}.

\paragraph{Example 1:} If $E(e_{ti}e_{tj})=G_{ij}$ and $e_{ti}e_{tj}$ are independent, then by central limit theorem, $\frac{1}{n}\sum_{t=1}^{n}e_{ti}e_{tj}-G_{ij}=O_p(\frac{1}{\sqrt{n}})$ if $\text{Var}(e_{ti}e_{tj})\leq M, \forall i,j \in J$. In such case, $\norm{\frac{1}{n}\mathbf{E}_J^T\mathbf{E}_J-\mathbf{G}_J}^2\leq O_p(\frac{m^2}{n})$ converges to 0 in probability be Assumption A2. If $\{e_{ti}e_{tj}-G_{ij}\}$ is a martingale sequence instead of independent, the argument still holds with martingale central limit theorem and uniformly bounded conditional variance; See Appendix A in Lai and Xing (2008).

\paragraph{Example 2:} If $\varepsilon_t$ follow GARCH(1,1) model in \eqref{Garch Model} with $E\sigma_t^4\leq M$, then $E(\frac{1}{\sqrt{n}}\sum_{t=1}^{n}\varepsilon_t e_{tj})^2=\frac{1}{n}E(\sum_{t=1}^{n}e_{tj}^2\sigma_t^2\xi_t^2+\sum_{s\neq t}e_{tj}\sigma_t e_{sj}\sigma_s \xi_t \xi_s)\leq \frac{1}{2n}\sum_{t=1}^{n}(Ee_{tj}^4+ E\sigma_t^4)\leq M$ by Assumption B3.1. This implies Assumption C3 holds for heteroskedastic $\varepsilon_t$. Similarly, Assumption C2 and C4 hold for heteroskedastic $e_{tj}$.\\

Under Assumptions A1-A3, B1-B4 and C1-C4, we have the following theorems.
\begin{theorem}
	$\frac{1}{n}\tilde{\mathbf{X}}_J^T\tilde{\mathbf{X}}_J\rightarrow \mathbf{G}_J$ in probability as $n\rightarrow \infty$
\end{theorem}

\begin{theorem}
	Given $\hat{k}=r$,$\norm{\tilde{\boldsymbol{\beta}}_J^0-\boldsymbol{\beta}_J}\leq O_p(\sqrt{\frac{m}{n}})$	
\end{theorem}

\begin{lemma}
	Under the assumptions for Theorems 2 and 4, $\tilde{\mathbf{w}}_J=\mathbf{Y}-\mathbf{X}_J\tilde{\boldsymbol{\beta}}_J^0$ converges to $\mathbf{w}_J=\mathbf{Y}-\mathbf{X}_J\boldsymbol{\beta}_J$ in distribution.
\end{lemma}

\subsection{$\mathbf{w}^{(b)}$ generation}
The results in Section 2.1 hold under the assumption that $J$ is independent of $\mathbf{X}$ and $\mathbf{Y}$. In order to handle the situation that $\hat{J}$ is selected by applying OGA on $\mathbf{X}$ and $\mathbf{Y}$, we divide the observations $\{\mathbf{x}_t,y_t\},t=1,\ldots,n$ into a training set $S^{\text{train}}=\{(\mathbf{x}_t,y_t):t=1,\ldots,[\frac{n}{2}]\}$ and a test set $S^{\text{test}}=\{(\mathbf{x}_t,y_t):t=[\frac{n}{2}]+1,\ldots,n\}$, where $[x]$ denotes the integer part of $x$. We apply OGA on $S^{\text{train}}$ to select $\hat{J}^{\text{train}}$, and then apply Algorithm 2 on $S^{\text{test}}$ and $\hat{J}^{\text{train}}$ to get $\tilde{\boldsymbol{\beta}}_{\hat{J}}^{\text{test}}$. If $S^{\text{train}}$ and $S^{\text{test}}$ are independent, $\tilde{\boldsymbol{\beta}}_{\hat{J}}^{\text{test}}$ is a consistent estimator of $\beta_j, j \in \hat{J}^{\text{train}}$ by the results in Section 2.1. Similarly, by exchanging the role of $S^{\text{train}}$ and $S^{\text{test}}$, we can get a consistent estimate $\tilde{\boldsymbol{\beta}}_{\hat{J}}^{\text{train}}$ for $\beta_j, j \in \hat{J}^{\text{test}}$, which is selected by applying OGA on $S^{\text{test}}$. Theorem 3 in Ing and Lai (2011) shows that, if all relevant variables satisfy $\beta_j^2\text{Var}(x_{tj})\gg n^{-\gamma}$ with $0\ll \gamma <1$ such that  $n^{2\gamma-1}\log p\rightarrow 0$, then $\hat{J}$ contains all relevant variables with probability approaching 1 as $n\rightarrow \infty$. This suggests that $\hat{J}^{\text{train}},\hat{J}^{\text{test}}$ and $\hat{J}$ select the 
same set of strong signals with high probability as $n\rightarrow 
\infty$. Hence we define an estimate $\tilde{\boldsymbol{\beta}}_{\hat{J}}$ for $\boldsymbol{\beta}_{\hat{J}}$ as follows: For $j \in \hat{J}$, if $j \in {\hat{J}}^{\text{train}}\cup \hat{J}^{\text{test}}$, let $\tilde{\boldsymbol{\beta}}_{\hat{J},j}$ be the corresponding estimate in $\hat{J}^{\text{train}}\cup \hat{J}^{\text{test}}$, i.e $\tilde{\boldsymbol{\beta}}_{\hat{J},j}^{\text{train}}$ or $\tilde{\boldsymbol{\beta}}_{\hat{J},j}^{\text{test}}$ or $(\tilde{\boldsymbol{\beta}}_{\hat{J},j}^{\text{train}}+\tilde{\boldsymbol{\beta}}_{\hat{J},j}^{\text{test}})/2$; if $j\notin \hat{J}^{\text{train}}\cup \hat{J}^{\text{test}}$, set $\tilde{\boldsymbol{\beta}}_{\hat{J},j}=0$. Let $\hat{J}_+=\{j\in \hat{J}:|\tilde{\boldsymbol{\beta}}_{\hat{J},j}|>0\}$ and consider the residuals $\tilde{\mathbf{w}}_{\hat{J}}=\mathbf{Y}-\mathbf{X}_{\hat{J}}\tilde{\boldsymbol{\beta}}_{\hat{J}}$. The standard bootstrapping residuals approach (see Section 1.6.2 of Lai and Xing (2008)), which resamples on $\tilde{\mathbf{w}}_{\hat{J}}$ to generate $\tilde{\mathbf{w}}_{\hat{J}}^{(b)}, b=1,\ldots,B$, may not be appropriate to be applied to generate $\mathbf{Y}^{(b)}=\mathbf{X}_{\hat{J}}\tilde{\boldsymbol{\beta}}_{\hat{J}}+\tilde{\mathbf{w}}_{\hat{J}}^{(b)}$ for two major issues. First, it ignores the potential correlation between $\mathbf{X}_{\hat{J}}$ and the residual $\mathbf{w}_{\hat{J}}=\mathbf{Y}-\mathbf{X}_{\hat{J}}\boldsymbol{\beta}_{\hat{J}}$. Such correlation may not be weak if $\mathbf{X}_j$ are correlated and there are some $j \notin \hat{J}$ with $|\beta_j|>0$ . Second, $\mathbf{Y}^{(b)}$ is independent of $\mathbf{X}_j$ for $j \notin \hat{J}$ but $\mathbf{Y}$ can have significant correlation with $\mathbf{X}_j$ if $|\beta_j|>0$. Note that $\tilde{\mathbf{w}}_{\hat{J}}$ is approximately equal to $\mathbf{Y}-\mathbf{X}_{\hat{J}_+}\boldsymbol{\beta}_{\hat{J}_+}= \sum_{j \notin \hat{J}_+}\beta_j\mathbf{X}_j+\boldsymbol{\varepsilon}$ as $\tilde{\boldsymbol{\beta}}_{\hat{J}_+}$ converges to $\boldsymbol{\beta}_{\hat{J}_+}$ asymptotically by Theorem 4, and both issues come from the case that there are significantly non-zero $\beta_j$ for $j \notin J_+$. To alleviate both issues, we want to get a good estimate for $\mathbf{X}_{\hat{J}_+^c}\boldsymbol{\beta}_{\hat{J}_+^c}=\sum_{j \notin \hat{J}_+}\beta_j\mathbf{X}_j$. However, since those variables are not selected at the first place by OGA, applying OGA on $\tilde{\mathbf{w}}_{\hat{J}}$ and $\mathbf{X}_{\hat{J}_+^c}$ directly is usually inefficient. By the assumptions that $\mathbf{X}_{\hat{J}_+^c}=\mathbf{F}\mathbf{\Lambda}_{\hat{J}_+^c}+\mathbf{E}_{\hat{J}_+^c}$ and the factor matrix $\mathbf{F}$ is weakly dependent with $\mathbf{E}_{\hat{J}_+^c}$ and $\boldsymbol{\varepsilon}$, we decompose $\mathbf{X}_{\hat{J}_+^c}\boldsymbol{\beta}_{\hat{J}_+^c}$ into $\mathbf{F}\mathbf{\Lambda}_{\hat{J}_+^c}\boldsymbol{\beta}_{\hat{J}_+^c}+\mathbf{E}_{\hat{J}_+^c}\boldsymbol{\beta}_{\hat{J}_+^c}$ and the part $\mathbf{E}_{\hat{J}_+^c}\boldsymbol{\beta}_{\hat{J}_+^c}$ can be approximated by a linear combination of the columns of $(\mathbf{I}-\mathbf{P}_{
	\hat{F}})\mathbf{X}_{\hat{J}_+^c}$, where $\hat{\mathbf{F}}$ is an estimate of $\mathbf{F}$ computed in Algorithm 1, due to the weak dependence assumption between $\mathbf{F}$ and $\mathbf{E}$. We then apply OGA on $\tilde{\mathbf{w}}_{\hat{J}}$ and $[\hat{\mathbf{F}},(\mathbf{I}-\mathbf{P}_{
	\hat{F}})\mathbf{X}_{\hat{J}_+^c}]$ to get an estimate $\hat{\boldsymbol\varepsilon}$ of $\boldsymbol\varepsilon$, and resample on $\hat{\boldsymbol\varepsilon}$ to generate $\mathbf{w}^{(b)}=\tilde{\mathbf{w}}_{\hat{J}}-\hat{\boldsymbol{\varepsilon}}+\hat{\boldsymbol{\varepsilon}}^{(b)}, b=1,\ldots B$. Since $\varepsilon_t$ in the martingale regression model \eqref{Linear Model} can be dependent, we apply the double block bootstrap method in Lee and Lai (2009) to handle the dependent data. We present the double block bootstrap method in Algorithm 3 and the summary of our procedure for $\mathbf{w}^{(b)}$ generation in Algorithm 4.

\begin{algorithm}[!h]
	\begin{algorithmic}[]		
		\State \textbf{INPUT}: $\boldsymbol{\varepsilon}=(\varepsilon_1,\ldots,\varepsilon_n)\in \mathbb{R}^n$
		\State \textbf{Step 1}: Resampling the first-level block bootstrap series $\mathcal{X}^*=(\varepsilon^*_1,\ldots,\varepsilon_n^*)$
		\begin{enumerate}
			\item[1.1] Let blocks $B_{j,l}=(\varepsilon_j,\varepsilon_{j+1},\ldots,\varepsilon_{j+l-1}), j=1,\ldots,n'$ be the overlapping blocks where $l=[n^{1/3}]$, $n'=n+l-1$.
			\item[1.2] Sampling $a=[n/l]$ blocks randomly with replacement from $\{B_{j,l},j=1,\ldots,n'\}$ and pasting them end to end to get $\mathcal{X}^*$.
		\end{enumerate}
		\State \textbf{Step 2}: Resampling the second-level block bootstrap series $\mathcal{X}^{**}$.
		\begin{enumerate}
			\item[2.1] Let blocks $B_{i,j,k}^*=(\varepsilon^*_{(i-1)l+j},\varepsilon^*_{(i-1)l+j+1},\ldots,\varepsilon^*_{(i-1)l+j+k-1})$ be the overlapping blocks within the block $(\varepsilon_{(i-1)l+1},\ldots,\varepsilon_{il})$, where $k=[l/2]$, $i=1,\ldots,a$, $j=1,\ldots,l'$, $l'=l-k+1$.
			\item[2.2] Sampling $c=[n/k]$ blocks with replacement from $\{B_{i,j,k}^*,i=1,\ldots,a,j=1,\ldots,l'\}$ and pasting them end to end to get $\mathcal{X}^{**}$.
		\end{enumerate}
	\State \textbf{OUTPUT}: $\boldsymbol{\varepsilon}^{(b)}=\mathcal{X}^{**}$.
	\end{algorithmic}
\caption{ $\boldsymbol{\varepsilon}^{(b)}$ generation }
\end{algorithm}

\begin{algorithm}[!h]
	\begin{algorithmic}[]
		
		\State \textbf{INPUT}: $\mathbf{X}\in \mathbb{R}^{n\times p}$, $\mathbf{Y} \in \mathbb{R}^n$
		\State \textbf{Step 1}: Apply OGA and Algorithm 1 on $(\mathbf{X},\mathbf{Y})=\{(\mathbf{x}_t,y_t), t=1,\ldots\}$ to select $\hat{J}$ and compute $\hat{\mathbf{F}}$.
		
		\State \textbf{Step 2}: Compute $\tilde{\boldsymbol{\beta}}_{\hat{J}}$
		\begin{enumerate}
			\item [2.1] Divide $(\mathbf{X},\mathbf{Y})$ into $(\mathbf{X}^{\text{train}},\mathbf{Y}^{\text{train}})=\{(\mathbf{x}_t,y_t):t=1,\ldots,[\frac{n}{2}]\}$ and $(\mathbf{X}^{\text{test}},\mathbf{Y}^{\text{test}})=\{(\mathbf{x}_t,y_t):t=[\frac{n}{2}]+1,\ldots,n\}$. Apply OGA on $(\mathbf{X}^{\text{train}},\mathbf{Y}^{\text{train}})$ and $(\mathbf{X}^{\text{test}},\mathbf{Y}^{\text{test}})$ to select $\hat{J}^{\text{train}}$ and $\hat{J}^{\text{test}}$.
			\item [2.2] Apply Algorithm 2 on $(\mathbf{X}^{\text{train}},\mathbf{Y}^{\text{train}},\hat{J}^{\text{train}})$ and $(\mathbf{X}^{\text{test}},\mathbf{Y}^{\text{test}},\hat{J}^{\text{test}})$ to compute $\tilde{\boldsymbol{\beta}}^{\text{train}}_{\hat{J}}$ and $\tilde{\boldsymbol{\beta}}^{\text{test}}_{\hat{J}}$ for each $j \in \hat{J}$, and set
			\begin{equation}
			\tilde{\boldsymbol{\beta}}_{\hat{J},j}=
			\begin{cases}
			0 , & \text{if}\ j \notin \hat{J}^{\text{train}}\cup\hat{J}^{\text{test}}\\
			(\tilde{\boldsymbol{\beta}}_{\hat{J},j}^{\text{train}}+\tilde{\boldsymbol{\beta}}_{\hat{J},j}^{\text{test}})/2, & \text{if}\   j \in \hat{J}^{\text{train}}\cap\hat{J}^{\text{test}}\\
			\tilde{\boldsymbol{\beta}}_{\hat{J},j}^{\text{test}}, & \text{if}\ j \in \hat{J}^{\text{train}} \setminus \hat{J}^{\text{test}}\\
			\tilde{\boldsymbol{\beta}}_{\hat{J},j}^{\text{train}}, & \text{if}\ j \in \hat{J}^{\text{test}}\setminus \hat{J}^{\text{train}}
			\end{cases}
			\end{equation}		       	
			
		\end{enumerate}
		\State \textbf{Step 3}: Estimate $\boldsymbol{\varepsilon}$.
		\begin{enumerate}
			\item[3.1] Compute $\tilde{\mathbf{w}}_{\hat{J}}=\mathbf{Y}-\mathbf{X}_{\hat{J}}\tilde{\boldsymbol{\beta}}_{\hat{J}}$, and $\tilde{\mathbf{X}}^F=[\hat{\mathbf{F}},(\mathbf{I}-\mathbf{P}_{\hat{F}})\mathbf{X}_{\hat{J}_+^c}]$, where $\hat{J}_+=\{j\in \hat{J}:|\tilde{\boldsymbol{\beta}}_{\hat{J},j}|>0\}$. 
			Divide $(\tilde{\mathbf{X}}^F,\tilde{\mathbf{w}}_{\hat{J}})$ into $(\tilde{\mathbf{X}}^{F,\text{train}},\tilde{\mathbf{w}}^{\text{train}}_{\hat{J}})$ and $(\tilde{\mathbf{X}}^{F,\text{test}},\tilde{\mathbf{w}}^{\text{test}}_{\hat{J}})$ as in Step 2.1. Apply OGA on $(\tilde{\mathbf{X}}^{F,\text{train}},\tilde{\mathbf{w}}^{\text{train}}_{\hat{J}})$ and $ (\tilde{\mathbf{X}}^{F,\text{test}},\tilde{\mathbf{w}}^{\text{test}}_{\hat{J}})$ to select $\hat{J}^{\text{train}}_w$ and $\hat{J}^{\text{test}}_w$.
			\item[3.2] Let $\hat{J}_w=\hat{J}^{\text{train}}_w\cap \hat{J}^{\text{test}}_w$. 
			Compute $\hat{\boldsymbol{\varepsilon}}^{\text{train}} =\tilde{\mathbf{w}}^{\text{train}}_{\hat{J}}  -\tilde{\mathbf{X}}^{F,\text{train}}_{\hat{J}_w} \hat{\boldsymbol{\beta}}^{\text{train}}_{\hat{J}_w}$, where $\hat{\boldsymbol\beta}^{\text{train}}_{\hat{J}_w}$ is the sub-vector of $\hat{\boldsymbol\beta}^{\text{train}}_{\hat{J}^{\text{test}}}$, which is the OLS estimate from the regression of $\tilde{\mathbf{w}}^{\text{train}}$ on $\tilde{\mathbf{X}}^{F,\text{train}}_{\hat{J}^{\text{test}}_w}$. Similarly, compute $\hat{\boldsymbol{\varepsilon}}^{\text{test}}=\tilde{\mathbf{w}}^{\text{test}}_{\hat{J}}-\tilde{\mathbf{X}}^{F,\text{test}}_{\hat{J}_w}\hat{\boldsymbol{\beta}}^{\text{test}}_{\hat{J}_w}$.
		\end{enumerate}	
		\State \textbf{Step 4}: Resample $\hat{\boldsymbol{\varepsilon}}^{(b)}$ by Algorithm 3 from $\hat{\boldsymbol{\varepsilon}}=\hat{\boldsymbol{\varepsilon}}^{\text{train }}\cup \hat{\boldsymbol{\varepsilon}}^{\text{test}}$, for $b=1,\ldots,B$.
		\State \textbf{OUTPUT}: $\mathbf{w}^{(b)}=\tilde{\mathbf{w}}_{\hat{J}}-\hat{\boldsymbol{\varepsilon}}+\hat{\boldsymbol{\varepsilon}}^{(b)}, b=1,\ldots B$ and $\tilde{\boldsymbol{\beta}}_{\hat{J}}$
	\end{algorithmic}
	\caption{ $\mathbf{w}^{(b)}$ generation }
	\label{alg:algorithm 2}
\end{algorithm}

\section{Hybrid Resampling for confidence intervals}
By using the duality between hypothesis tests and confidence regions, we present our algorithm for constructing confidence intervals through a sequence of hypothesis tests. We first introduce the OGA  that have been considered for variable selection in this paper, then we define hypotheses and the corresponding test statistics that involve OGA in Section 3.2. Those test statistics will then be used to construct confidence intervals in Section 3.3.

\subsection{Orthogonal greedy algorithm}
Orthogonal greedy algorithm (OGA) is a  method that based on the framework of $L_2$-boosting procedure introduced by B{\"u}hlmann and Yu (2003) to select the input variables in linear regression in the case $p\gg n$. The OGA that present in this section is a modification of the one introduced by Ing and Lai (2011). This modification applies QR decomposition to improve efficiency, but the results are equivalent to those from original OGA in Ing and Lai (2011).

\begin{algorithm}[!h]
	\begin{algorithmic}[]
		\State \textbf{INPUT}: $\mathbf{X}\in \mathbb{R}^{n\times p}$, $\mathbf{Y} \in \mathbb{R}^n$.
		\State \textbf{Step 1}: Initialize $\mathbf{U}^{(0)}=\mathbf{Y}$, $\hat{J}_0=\emptyset$, empty matrix $\mathbf{Q}_0$ and $\mathbf{R}_0$.
		\State \textbf{Step 2} For $k=1$ to $m$
		\begin{enumerate}
			\item [2.1]Choose $\hat{j}_k \notin \hat{J}_{k-1}$ such that $\mathbf{X}_{\hat{j}_k}$ is most correlated to $\mathbf{U}^{(k-1)}$, i.e.
			$$\hat{j}_k=\arg\max_{j \notin \hat{J}_{k-1}}\|\mathbf{X}_j^T\mathbf{U}^{(k-1)}\|/\|\mathbf{X}_j\|.$$	
			\item [2.2] 	Update $\hat{J}_{k}=\hat{J}_{k-1}\cup \{\hat{j}_k\}$ and compute the QR decomposition 
			$$\mathbf{X}_{\hat{J}_k}=
			\begin{bmatrix} 
			\mathbf{X}_{\hat{J}_{k-1}}&\mathbf{X}_{\hat{j}_k} \end{bmatrix}=
			\begin{bmatrix} 
			\mathbf{Q}_{k-1}&\mathbf{q}_{k} 
			\end{bmatrix}
			\begin{bmatrix} \mathbf{R}_{k-1}&\mathbf{r}_{k,1}\\
			\mathbf{0}^T & r_{k,2} \end{bmatrix}
			=\mathbf{Q}_k\mathbf{R}_k$$
			\item [2.3] Update $\mathbf{U}^{(k)}=\mathbf{U}^{(k-1)}-q_k\beta_k^q$, where $\beta_k^q=q_k^T\mathbf{U}^{(k-1)}$.
		\end{enumerate}	
		\State \textbf{Step 3}: Compute $\hat{\boldsymbol\beta}^{\text{OGA}}$ with its $\hat{j}_k$th entry equals to the $k$th entry of $\mathbf{R}_m^{-1}(\beta_1^q,\ldots,\beta_m^q)^T$ and other entries equal to 0.
		\State \textbf{OUTPUT}: $\hat{J}=\hat{J}_m$, $\hat{\boldsymbol\beta}^{\text{OGA}} \in \mathbb{R}^p$.
	\end{algorithmic}
	\caption{Orthogonal Greedy Algorithm(OGA)}
	\label{alg:algorithm 2}
\end{algorithm}

Here $\mathbf{R}_m^{-1}(\beta_1^q,\ldots,\beta_m^q)^T$ can be computed by backward substitution without calculating the inverse of the upper triangular matrix $\mathbf{R}_m$, and the QR decomposition is used to implement forward stepwise regression, instead of sequentially orthogonalizing the input variables as in Section 2.2 of Ing and Lai (2011).

\subsection{Test statistic computation}
Suppose we want to test $\text{H}_0:\beta_j=\theta$ for a particular $j \in \{1,\ldots,p\}$ given observed samples $(\mathbf{X},\mathbf{Y})$. Note that we don't restrict $j \in \hat{J}$ as $\hat{J}$ is a random variable. Instead, we consider OGA selection as a part of the test statistic computation. Given $(\mathbf{X},\mathbf{Y})$, we first conduct OGA to select $\hat{J}$. If $j\notin \hat{J}$, we simply set the test statistic $T_j(\mathbf{X},\mathbf{Y},\theta)=0$. For $j\in \hat{J}$, we compute $\tilde{\boldsymbol{\beta}}_{\hat{J}}^0$ by Algorithm 2 with $(\mathbf{X},\mathbf{Y})$ and $\hat{J}$, and we consider the asymptotic distribution of $\sqrt{n}(\tilde{\boldsymbol{\beta}}_{\hat{J}}^0-\boldsymbol{\beta}_{\hat{J}})$ under the assumptions that $\hat{J}$ is fixed and $E((y_t-\sum_{j\in \hat{J}}x_{tj}\beta_j)\tilde{\mathbf{x}}_{\hat{J},t}^T)=\mathbf{0}$, where $\tilde{\mathbf{x}}_{\hat{J},t}^T$ is the $t$th row of $\tilde{\mathbf{X}}_{\hat{J}}$ in model \eqref{Decompositon of Y_2} with $J=\hat{J}$. Although the distribution does not hold due to the violation of the assumptions, our resampling approach does not require the knowledge of the true distribution of the test statistics. With the assumption $\hat{J}$ is fixed and other assumptions stated in Section 9.7 of Lai and Xing (2008), $\sqrt{n}(\tilde{\boldsymbol{\beta}}_{\hat{J}}^0-\boldsymbol{\beta}_{\hat{J}})$ has a limiting $N(0,\mathbf{V})$, where

\begin{align}
\mathbf{V}=n(\tilde{\mathbf{X}}_{\hat{J}}^T\tilde{\mathbf{X}}_{\hat{J}})^{-1}\mathbf{S}(\tilde{\mathbf{X}}_{\hat{J}}^T\tilde{\mathbf{X}}_{\hat{J}})^{-1}, \label{V}
\end{align}

\noindent and $S$ is suggested in Lai and Xing (2008) to be 
\begin{align}
\mathbf{S}=\tilde{\mathbf{X}}_{\hat{J}}^T\begin{pmatrix} \hat{w}_1& \ldots& 0\\   & \ddots& \\    0& \ldots& \hat{w}_n \end{pmatrix}\tilde{\mathbf{X}}_{\hat{J}}=\tilde{\mathbf{X}}_{\hat{J}}^T\text{diag}(\mathbf{Y}-\mathbf{X}_{\hat{J}}\tilde{\boldsymbol{\beta}}_{\hat{J}}^0)\tilde{\mathbf{X}}_{\hat{J}}\label{S1 for V}
\end{align}
if $(\mathbf{x}_t,\varepsilon_t)$ in model \eqref{Linear Model} are uncorrelated. For correlated $(\mathbf{x}_t,\varepsilon_t)$, let $\hat{\mathbf{\Gamma}}_0$ denote the right-hand side of \eqref{S1 for V} and define for $\nu\geq 1$.
\begin{align*}
\hat{\mathbf{\Gamma}}_{\nu}=\sum_{t=\nu+1}^{n}(\mathbf{g}_t\mathbf{g}_{t-\nu}^T+\mathbf{g}_{t-\nu}\mathbf{g}_t^T),\quad \text{where } \mathbf{g}_t=(y_t-\mathbf{X}^T_{t\hat{J}}\tilde{\boldsymbol{\beta}}_{\hat{J}}^0)\tilde{\mathbf{x}}_{\hat{J},t}
\end{align*}
then $\mathbf{S}$ is suggested to be 
\begin{align}
\mathbf{S}=\hat{\mathbf{\Gamma}}_0+\sum_{\nu=1}^{q}\big(1-\frac{\nu}{q+1}\big)\hat{\mathbf{\Gamma}}_{\nu}\label{S2 for V}
\end{align}
in which $q\rightarrow \infty$ but $q/n^{1/4}\rightarrow 0$ as $n\rightarrow \infty$. Algorithm 6 summarizes our procedure for computing test statistic $T_j(\mathbf{X},\mathbf{Y},\theta)$ for testing $H_0:\beta_j=\theta$.

\begin{algorithm}[!h]
	\begin{algorithmic}[]
		\State \textbf{INPUT}: $\mathbf{X}\in \mathbb{R}^{n\times p}$, $\mathbf{Y} \in \mathbb{R}^n$, $\theta$.
		\State \textbf{Step 1}: Apply OGA on $(\mathbf{X},\mathbf{Y})$ to select $\hat{J}$.
		\State \textbf{Step 2}: If $j \notin \hat{J}$,  STOP and OUTPUT $T_j(\mathbf{X},\mathbf{Y},\theta)=0$.
		\State \textbf{Step 3}: Compute $\tilde{\boldsymbol{\beta}}_{\hat{J}}$ and $\mathbf{X}_{\hat{J}}$ by Algorithm 2 given $(\mathbf{X},\mathbf{Y})$ and $\hat{J}$.
		\State \textbf{Step 4}: Compute $\mathbf{V}=n(\tilde{\mathbf{X}}_{\hat{J}}^T\tilde{\mathbf{X}}_{\hat{J}})^{-1}\mathbf{S}(\tilde{\mathbf{X}}_{\hat{J}}^T\tilde{\mathbf{X}}_{\hat{J}})^{-1}$, and 
		$$T_j=\frac{\tilde{\boldsymbol{\beta}}_{\hat{J},j}-\theta}{\sqrt{\frac{1}{n}V_{jj}}}$$
		where $V_{jj}$ is the $j$th diagonal entry of $V$ and $S$ is computed by \eqref{S1 for V} or \eqref{S2 for V}.
		\State \textbf{OUTPUT}: $T_j(\mathbf{X},\mathbf{Y},\theta)=|T_j|$ 
	\end{algorithmic}
	\caption{Test statistic for testing $H_0: \beta_j=\theta$ }
\end{algorithm}

\subsection{Confidence intervals by hybrid resampling}
For a particular $j\in \hat{J}$, suppose we know everything about model \eqref{Linear Model} except $\beta_j$. Let
\begin{align}
y_t(j,\theta)=\sum_{i\neq j} x_{ti}\beta_i+x_{tj}\theta+\varepsilon_t,\quad t=1,\ldots,n \label{linear model with theta}
\end{align}
and $\mathbf{Y}(j,\theta)=(y_1(j,\theta),\ldots,y_n(j,\theta))^T$. Then $\mathbf{Y}(j,\beta_j)$ has the same distribution as the observed $\mathbf{Y}$, and hence $T_j(\mathbf{X},\mathbf{Y}(j,\beta_j),\beta_j)$ conditioned on $T_j(\mathbf{X},\mathbf{Y}(j,\beta_j),\beta_j)>0$ also has the same distribution as $T_j(\mathbf{X},\mathbf{Y},\beta_j)$ conditioned on $T(\mathbf{X},\mathbf{Y},\beta_j)>0$. Let $u_{\alpha}(\theta)$ be the $\alpha$-quantile of the distribution of $T_j(\mathbf{X},\mathbf{Y}(j,\theta),\theta)$ given that $T_j(\mathbf{X},\mathbf{Y}(j,\theta),\theta)>0$. It can be computed by simulating $\mathbf{Y}^{(b)}(j,\theta)$, and the corresponding $T_j(\mathbf{X},\mathbf{Y}^{(b)}(j,\theta),\theta)$, for $b=1,\ldots,B$ from \eqref{linear model with theta}. Then, for $j \in \hat{J}$, we have 
\begin{align}
\mathbf{P}(u_{\alpha}(\theta)<T_j(\mathbf{X},\mathbf{Y},\theta)<u_{1-\alpha}(\theta)|\theta=\beta_j)=1-2\alpha. \label{Condition Probabilty }
\end{align}
From \eqref{Condition Probabilty }, we have $\mathbf{P}(\beta_j \in C_{\alpha}^i(\mathbf{X},\mathbf{Y}))=1-2\alpha$ for $j \in \hat{J}$, where 
\begin{align}
C_{\alpha}^j(\mathbf{X},\mathbf{Y})=\{\theta:u_{\alpha}(\theta)<T_j(\mathbf{X},\mathbf{Y},\theta)<u_{1-\alpha}(\theta)\} \label{C_alpha}
\end{align}
Therefore, $C_{\alpha}^j(\mathbf{X},\mathbf{Y})$ is a $100(1-2\alpha)\%$ confidence region for $\beta_j$ for $j\in \hat{J}$. In practice, we cannot simulate $\mathbf{Y}^{(b)}(j, \theta)$ from \eqref{linear model with theta}. Instead, we consider model \eqref{Linear Model decomposition in Capital} with $\boldsymbol{\beta}_{\hat{J}}$ and $\boldsymbol{w}$ replaced by $\tilde{\boldsymbol{\beta}}_{\hat{J}}$ and $\boldsymbol{w}^{(b)}$ that are generated by Algorithm 4 as an approximated model for $\mathbf{Y}$. That is , we assume 
\begin{align}
\hat{\mathbf{Y}}^{(b)}(j,\tilde{\boldsymbol{\beta}}_{\hat{J},j})=\mathbf{X}_{\hat{J}}\tilde{\boldsymbol{\beta}}_{\hat{J}}+\mathbf{w}^{(b)},b=1,\ldots,B \label{Resample Y}
\end{align}
have a similar distribution with $\mathbf{Y}(j,\tilde{\boldsymbol{\beta}}_{\hat{J},j})$. For general $\theta$, we replace the coefficient of $\mathbf{X}_j$ in \eqref{Resample Y} by $\theta$ and get 
\begin{align}
\hat{\mathbf{Y}}^{(b)}(j,\theta)=\hat{\mathbf{Y}}^{(b)}(j,\tilde{\boldsymbol{\beta}}_{\hat{J},j})+(\theta-\tilde{\boldsymbol{\beta}}_{\hat{J},j})\mathbf{X}_j
\end{align}
Replacing $\mathbf{Y}^{(b)}(j,\theta)$ by $\hat{\mathbf{Y}}^{(b)}(j,\theta)$ in \eqref{C_alpha}, with $\hat{u}_{\alpha}(\theta)$ being the $\alpha$-quantile of $T_j(\mathbf{X},\hat{\mathbf{Y}}^{(b)}(j,\theta),\theta)$ given that $T_j(\mathbf{X},\hat{\mathbf{Y}}^{(b)}(j,\theta),\theta)>0$, then an approximated $100(1-2\alpha)\%$ confidence region for $\beta_j$ for $j \in \hat{J}$ is given by 
\begin{align}
\hat{C}_{\alpha}^j(\mathbf{X},\mathbf{Y})=\{\theta:\hat{u}_{\alpha}(\theta)<T_j(\mathbf{X},\mathbf{Y},\theta)<\hat{u}_{1-\alpha}(\theta)\}. \label{C_alpha_hat}
\end{align}
Although $\hat{C}_{\alpha}^j(\mathbf{X},\mathbf{Y})$ may not be an interval, it is often suffices to give only the upper and lower limits of $\hat{C}_{\alpha}^j(\mathbf{X},\mathbf{Y})$ to construct a confidence interval. Our resampling approach for constructing confidence intervals is essentially the same as the hybrid resampling approach introduced by Chuang and Lai (2000). While they consider unconditional confidence intervals, we consider confidence intervals given that $j\in \hat{J}$, or equivalently $T_j(\mathbf{X},\mathbf{Y},\beta_j)>0$. We present our main algorithm for constructing confidence intervals for $\beta_j: j\in\hat{J}$ in Algorithm 7.

\begin{algorithm}[!h]
	\begin{algorithmic}[]
		\State \textbf{INPUT}: $\mathbf{X}\in \mathbb{R}^{n\times p}$, $\mathbf{Y} \in \mathbb{R}^n$, 
		\State \textbf{Step 1}: Apply OGA on $(\mathbf{X},\mathbf{Y})$ to select $\hat{J}$.
		\State  \textbf{Step 2}: Apply Algorithm 4 on $(\mathbf{X},\mathbf{Y})$ to generate $\tilde{\boldsymbol{\beta}}_{\hat{J}}$ and $\mathbf{w}^{(b)},b=1,\ldots,B$.
		\State  \textbf{Step 3}:For each $j\in \hat{J}$,
		\begin{enumerate}
			\item [3.1] For a grid of $\theta$, compute $T_j((\mathbf{X},\mathbf{Y}),\theta)$ by Algorithm 6, $\hat{u}_{\alpha}(\theta)$ and $\hat{u}_{1-\alpha}(\theta)$ in \eqref{C_alpha_hat}.
			\item [3.2] Find $\theta_{u}^j=\min_{\theta}|\hat{u}_{\alpha}(\theta)-T_j((\mathbf{X},\mathbf{Y}),\theta)|$.
			\item [3.3] Find $\theta_{l}^j=\min_{\theta}|\hat{u}_{1-\alpha}(\theta)-T_j((\mathbf{X},\mathbf{Y}),\theta)|$.
		\end{enumerate}
		\State \textbf{OUTPUT}: $(\theta_{l}^j,\theta_{u}^j)$ as the estimated confidence interval for $\beta_j,j\in\hat{J}$	
	\end{algorithmic}
	\caption{Hybrid resampling confidence intervals for selected coefficients}
\end{algorithm}

While Algorithm 7 estimates $100(1-2\alpha)\%$ two-sided confidence intervals for $\beta_j$, it is strict forward to modify it for one-sided intervals. Suppose we want to find a confidence interval of $\beta_j$ of the form $(\theta_{l}^j,\infty)$. We can modify Algorithm 6 to set $T_j(\mathbf{X},\mathbf{Y},\theta)=-\infty$ if $i\notin \hat{J}$ and $T_i(\mathbf{X},\mathbf{Y},\theta)=T_i$ in Step 4 of Algorithm 6 if $i \in \hat{J}$. Then $\hat{u}_{1-\alpha}(\theta)$ in Step 3.1 of Algorithm 7 is the $(1-\alpha)$-quantile of $\{T_j(\mathbf{X},\mathbf{Y}^{(b)}(j,\theta),\theta)\mid T_j(\mathbf{X},\mathbf{Y}^{(b)}(j,\theta),\theta)>-\infty\}$, and we can simply set $\theta_u^j=\infty$. For $\theta_l^j$ in Step 3.3 of Algorithm 7, we can compute it by bisection method. First, we find $a_1$ such that $\hat{u}_{1-\alpha}(a_1)>T_j(\mathbf{X},\mathbf{Y},a_1)$. Usually, we can choose $a_1=\tilde{\boldsymbol{\beta}}_{\hat{J},j}$. Then, we find $r_1$ such that $\hat{u}_{1-\alpha}(r_1)<T_j(\mathbf{X},\mathbf{Y},r_1)$. To find $r_1$, one can start with $r'_1=a_1-2\hat{\sigma}_j$, where $\hat{\sigma}_j=\sqrt{V_{jj}/n}$ that is computed in Step 4 of Algorithm 6. If $\hat{u}_{1-\alpha}(r'_1)< T_j(\mathbf{X},\mathbf{Y},r'_1)$, set $r_1=r'_1$; otherwise let $r'_2=r'_1-\hat{\sigma}_j/2$ and check if $\hat{u}_{1-\alpha}(r'_2)<T_j(\mathbf{X},\mathbf{Y},r'_2)$. This procedure is repeated until one arrives at $\hat{u}_{1-\alpha}(r'_h)<T_j(\mathbf{X},\mathbf{Y},r'_h)$ and sets $r_1=r'_h$. Let $m_1=(a_1+r_1)/2$, if $\hat{u}_{1-\alpha}(m_1)>T_j(\mathbf{X},\mathbf{Y},m_1)$, set $a_2=m_1$ and $r_2=r_1$; otherwise set $a_2=a_1$ and $r_2=m_1$. This procedure is repeated until $a_k-r_k$ is smaller than some threshold $\delta$ or if $k$ reaches some upper bound, and $\theta_l^j$ is chosen to be $m_k=(a_k+r_k)/2$.

\section{Simulation Studies}\label{chapter04}
In this section, we illustrate the convergence rate of our $\mathbf{\beta}_{\hat{J}}$ estimator as stated in Theorem 4 and compare the performance of our approach with other existing methods through simulations under various settings. For all the simulations, we choose the maximum number $k_{\max}$ of factors in Algorithm 1 to be 5 and the number of resampling $B=50$ in Algorithm 4. For the number $m$ of OGA iterations, Ing and Lai (2011) show that the convergence of OGA estimator when $m=O(\sqrt{n/\log p})$. However, for finite $n$, the performance of OGA can be very different for different choices of $m$. As Ing (2019) point out that ``the approximation error decreases as the number $m$ of iterations increases and the sampling variability increases with $m$", and an optimal $m$ to balance such two terms is hard to determine ``because not only does the solution (optimal $m$) involve unknown parameters$\ldots$ but it is unknown which kind of sparsity holds." To overcome this difficulty, Ing (2019) propose a data-driven method to determine $m$. We follow the idea and use HDBIC in Ing and Lai (2011) to choose $m$. Using the notation in OGA presented in Algorithm 5, we define
$$\text{HDBIC}(k)=n\log\|\mathbf{U}^{(k)}\|^2+k\log n\log p,$$
and take $m=\arg\min_{1\leq k\leq K_n}\text{HDBIC}(k)$, where $K_n=2[\sqrt{n/\log p}]$. Note that we first do $K_n$ OGA iterations in Algorithm 5, and then determine $m$ by HDBIC. With the further selection by HDBIC, the choice of $K_n$ is not sensitive to the final $m$.

We consider the martingale regression models \eqref{Linear Model} with 4 different choices of $x_{tj}$ and $\varepsilon_t$.
\begin{enumerate}
	\item \textbf{LAI}. It is the same as that in the example 1 of Ing and Lai(2011), where $x_{tj}=f_t+e_{tj}$ with $f_t$ and $e_{tj}$ are i.i.d.\ standard normal. The errors $\varepsilon_t$ are also i.i.d.\ standard normal. 
	\item \textbf{GARCH}. The errors follow GARCH(1,1) with $\varepsilon_t=\sigma_t\xi_t$, $\sigma_t^2=0.1+0.3\sigma_{t-1}^2+0.3\varepsilon_{t-1}^2$ and $\xi_t$ are i.i.d. $N(0,1)$. For the components of $\mathbf{x}_t$, $x_{tj}=f_t(1+|a_{j}|)+e_{tj}$ with $f_t=0.9f_{t-1}+b_t$, where $a_j$, $e_{tj}$ and $b_t$ are i.i.d.\ standard normal.
	\item \textbf{AR}. The response $y_t$ is related to $y_{t-1}$ by replacing $x_{t1}$ in GARCH setting by $y_{t-1}$. That is,
	$$y_t=\beta_1 y_{t-1}+\sum_{j=2}^p\beta_j x_{tj}+\varepsilon_t.$$
	The predictors $x_{tj}$ are the same as in GARCH setting, and $\varepsilon_t$ are i.i.d.\ standard normal.
	\item \textbf{IID}. The predictors $x_{tj}\sim N(0,2)$ i.i.d.\ and the errors $\varepsilon_t\sim N(0,1)$ i.i.d.
	\item \textbf{MVN}. The rows of $\mathbf{X}$ follow i.i.d.\ multivariate normal distribution $N(0,\Sigma)$, where $\Sigma_{jk}=0.2$ if $j\neq k$ otherwise $\Sigma_{jk}=1$. The errors $\varepsilon_t$ are i.i.d.\ standard normal.
\end{enumerate}    
For all settings, we consider the same $\boldsymbol{\beta}=(\beta_1,\ldots,\beta_p)$, which has the first 10 entries nonzero. We set $\beta_1=\beta_2=0.6$, $\beta_3=0.4$, $\beta_5=\beta_6=\beta_7=0.2$, $\beta_7$ to $\beta_{10}$ equal to 0.1 and the remaining $\beta_j$ are zeros.  

\subsection{Convergence rate of $\tilde{\mathbf{\beta}}_{\hat{J}}$}

We conduct 2000 simulations to verify the convergence rate of our estimator. In the $l$th simulation, let $\tilde{\boldsymbol{\beta}}^{(l)}_{\hat{J}}$ be our estimate in Algorithm 4 with $\hat{J}^{(l)}$ of size $m^{(l)}$ selected and $\boldsymbol{\beta}_{\hat{J}^{(l)}}$ be the corresponding true values. Define the square-root of the mean squared error for the $l$th simulation to be 
\begin{align*}
\sqrt{\text{MSE}}^{(l)}=\sqrt{\frac{1}{m^{(l)}}\norm{\tilde{\boldsymbol{\beta}}_{\hat{J}}^{(l)}-\boldsymbol{\beta}_{\hat{J}^{(l)}}}}
\end{align*}
and we measure the performance of our estimates by average mean squared error
\begin{align*}
\text{AMSE}=\frac{1}{L}\sum_{l=1}^{L}\sqrt{\text{MSE}}^{(l)}.
\end{align*}
From Theorem 4, the convergence rate of $\text{AMSE}$ should be of order $1/\sqrt{n}$, which means $\text{AMSE}$ is expected to be reduced by half when we increase $n$ to $4n$. The results in Table 1 support Theorem 4.

\begin{table}[!h]
	\centering
	\caption{Average mean squared errors of the $\boldsymbol{\beta}_{\hat{J}}$ estimator in Algorithm 4  }
	\begin{tabular}{|c c c c c c |}
		\hline
		$(n,p)$& LAI  &  GARCH  & AR & IID    & MVN                 \\ 
		$(200,250)$	& 0.1240   & 0.0931  &  0.1092             & 0.0583                &  0.1180                                      \\
		$(400,500)$&  0.0693   &  0.0557      &   0.0706               &    0.0393                 &    0.0680                                      \\ 
		$(800,1000)$	&    0.0462 & 0.0289     &   0.0434              &    0.0263                  &    0.0459                                \\ \hline
	\end{tabular}
\end{table}

\subsection{Comparison with existing methods}

In this subsection, we compare our hybrid resampling (HR) approach for computing one-sided $80\%$ confidence intervals for selected $\beta_j$ of the form $(\theta^j_l,\infty)$ with other existing methods. We use equation \eqref{S2 for V} with $q=1$ for computing test statistics. The modification of Algorithm 7 for computing one-sided confidence intervals is described in the last paragraph of Section 3.3. Let $\theta>0$ be a non-zero value of $\beta_j$ so that $\theta\in\{0.1,0.2,0.4,0.6\}$, and $p(\theta)$ be the number of $\beta_j$ equals to $\theta$. Therefore, we have $p(0.1)=4$, $p(0.2)=3$, $p(0.6)=2$ and $p(0.4)=1$. We define $\text{NS}(\theta)=\frac{1}{p(\theta)}\sum_{\beta_j=\theta}\frac{\sum_{l=1}^{N} \mathbb{I}_{\{j\in \hat{J}^l\}}}{N}$ be the average number of selection for the set $\{j:\beta_j=\theta\}$ in $N=2000$ simulation. The performance for each method is measured by the coverage rate $\text{CR}(\theta)=\frac{1}{N\text{NS}(\theta)\text{p}(\theta)}\sum_{\{(j,l):\beta_j=\theta,j\in\hat{J}^l\}}\mathbb{I}_{\{\beta_j\geq \text{LB}\}}$, i.e. the ratio that the true parameter $\theta$ is covered by the estimated confidence intervals with the lower bound $\text{LB}$. The mean $mLB$ and the standard deviation sLB of LB are also reported. The results of comparison are presented in Tables 2 to 6. 

\paragraph{Classical $t$-distribution (t)}
This approach ignores the facts that the set $\hat{J}$ is selected and the coefficients of predictors not being selected can affect the estimates of $\beta_j$ for $j\in\hat{J}$. It uses the $t$-distribution in \eqref{t test statistic} with $\mathbf{X}$ replaced by $\mathbf{X}_{\hat{J}}$ to get the confidence interval. This approach only works well for the strong signals (with nearly 100\% selection rate, and thus negligible selection effect) in IID setting (no spill-over effect for strong signals). In such cases, our HR approach also works well. While the $t$-distribution approach shows significantly reduced coverage for other $\beta_j$, HR has much better performance. 

\paragraph{Instrumental variable (IV)}
This approach ignores the fact that the set $\hat{J}$ is selected. It uses Algorithm 2 to compute $\tilde{\boldsymbol\beta}^0_{\hat{J}}$, which has an approximate $N(\boldsymbol\beta_{\hat{J}},\mathbf{V}/n)$ distribution when $\hat{J}$ is fixed (i.e., no selection effect). Here $\mathbf{V}$ is defined in \eqref{V} with $\mathbf{S}$ chosen to be the one in \eqref{S2 for V} with $q=1$. Actually, considering the time series structures of $x_{tj}$ and $\varepsilon_t$ in GARCH and AR settings, we should choose a larger $q$. Here we mimic the case that $q$ is not well chosen. For fair comparison, HR also uses $q=1$ for constructing confidence intervals.

This approach and HR perform well for strong signals in LAI, IID and MVN settings. It is not surprising for LAI and IID settings, as LAI setting satisfies factor-model assumption A1, and IID setting has no spill-over effect for strong signals. The good performance for strong signals in MVN setting shows that our estimator in Algorithm 2 can handle spill-over effect even if $x_{tj}$ do not follow factor model. Note that the poor performance in MVN setting for post-selection inference (PS) approach, which is described later in this section, indicates that there is strong spill-over effect. 

However, for strong signals in GARCH setting, the coverage rates for IV approach are not close to nominal 0.8. It is because the variance estimates for $\boldsymbol\beta_{\hat{J}}$ are not good, and hence the distribution of the test statistics is no longer close to standard normal. Although a pivotal distribution of the test statistics is also important for our HR approach, HR approach can still work well if the underlying unknown distribution can be well approximated by the empirical distribution of the resampled test statistics. The performance of IV approach for strong signals improves in AR setting, in which the errors are i.i.d.\ instead of GARCH(1,1) in GARCH setting. However, the coverage rate of the coefficient for $y_{t-1}$ in AR setting is still significantly away from 0.8. It may be because $y_{t-1}$ does not follow the factor model and hence the distribution of the test statistic is not close to standard normal. Again, HR approach can still get a coverage rate close to 0.8 for the coefficient of $y_{t-1}$ in AR setting. For weak signals (0.1 and 0.2) that are not often selected by OGA, IV approach in general shows significantly reduced coverage due to the selection effect. HR has much better performance in all settings.

\paragraph{Post selection inference (PS)}
Taylor et al.\ (2014) derive an exact null distribution for their proposed test statistics after forward stepwise model selection in finite samples. They call such conditional inference as {\it post-selection} inference, which can be used to produce confidence intervals for appropriate underlying regression parameters. Using the notations in \eqref{Biase of estimator}, PS approach constructs confidence intervals for the entries of $E\hat{\boldsymbol\beta}^{\text{ols}}_{\hat{J}}$. Therefore, applying such confidence intervals for $\beta_j$, $j\in\hat{J}$, ignores the fact that $E\hat{\boldsymbol\beta}^{\text{ols}}_{\hat{J}}\neq \boldsymbol\beta_{\hat{J}}$ (spill-over effect). 

Let $\mathbf{v}_j$ be the $j$th column of $\mathbf{X}_{\hat{J}}(\mathbf{X}^T_{\hat{J}}\mathbf{X}_{\hat{J}})^{-1}$ and $F_{\mu,\sigma^2}^{[a,b]}$ denote the distribution function of a $N(\mu,\sigma^2)$ random variable truncated to lie in $[a,b]$, i.e.,
\begin{align*}
F_{\mu,\sigma^2}^{[a,b]}(x)=\frac{\Phi((x-\mu)/\sigma)-\Phi((a-\mu)/\sigma)}{\Phi((b-\mu)/\sigma)-\Phi((a-\mu)/\sigma)},
\end{align*}
where $\Phi$ is the distribution function of standard normal. If $\mathbf{Y}\sim N(\mathbf{X}\boldsymbol{\beta},\sigma^2\mathbf{I})$ with known $\sigma$, then Lemmas 1 and 2 of Taylor et al.(2014) show that there exist $\mathcal{V}^{\text{lo}}_j$ and $\mathcal{V}^{\text{up}}_j$ such that the solution $\delta_{\alpha}$ of the equation
\begin{align*}
1-F_{\delta_{\alpha},\sigma_2\|\mathbf{v}_j\|^2}^{[\mathcal{V}^{\text{lo}}_j,\mathcal{V}^{\text{up}}_j]}(\hat{\boldsymbol\beta}^{\text{ols}}_{\hat{J},j})=\alpha
\end{align*}
is the lower bound of the $100(1-\alpha)\%$ one-sided confidence interval for $E\hat{\boldsymbol\beta}^{\text{ols}}_{\hat{J},j}$.

The results show that PS approach does not work in our settings due to the spill-over effect. It only works for strong signals in IID setting. Since $\mathbf{X}_{\hat{J}}$ and $\mathbf{X}_j$ are independent for $j\notin \hat{J}$, the expected amount spilled-over on strong signals is expected to be 0 by \eqref{Biase of estimator} with $\mathbf{X}$ assumed to be random. However, it is not true for weak signals. If the amount spilled-over on a particular weak signal makes the signal even weaker, the corresponding predictor is unlikely to be selected. Therefore, given that $j\in \hat{J}$ and $\beta_j\approx 0$, the amount spilled-over on $\beta_j$ is likely to be of same sign of $\beta_j$, i.e., the expected amount of spill-over is not 0. It is the reason why PS approach does not work for weak signals in IID setting. Since $x_{tj}$ in IID setting do not satisfy Assumptions A1 and B1, HR approach does not help much in this case. However, HR performs much better than PS in all settings.\\       

To conclude, $t$-distribution performs worst among all the approaches, PS approach is not appropriate to construct confidence intervals for $\beta_j$ for $j\in\hat{J}$, IV approach does not work when $\beta_j\approx 0$, and HR approach performs the best in all the settings. Note that if $\beta_j$ are fixed, it will be selected by OGA with high probability (i.e., what we mean strong signal in this Section) as $n\rightarrow\infty$.

\begin{table}[!h]
	\centering
	\caption{Coverage rate and estimated mean lower bound and its standard deviation based on 4 different methods under the setting of \textbf{LAI}.}
	\begin{tabular}{|c c |c c c c c |}
		\hline
		\multicolumn{2}{|c|}{ $\beta_j$}		  & 0.6         & 0.4                                         & 0.2  & 0.1 & Overall \\ \hline
		\multicolumn{2}{|c|}{NS($n=200$)}		  &2000  &   1869    &   382   &   37.25  &          \\\hline
		CR	& t & 0.1033  &  0.0787       & 0.0787   & 0  &  0.1356          \\
		& IV&   0.8250 & 0.8384  & 0.4110  &0.0403   & 0.8182           \\
		&  PS& 0.2195  & 0.2980      &  0.3272  &  0.3020 & 0.2559           \\
		& HR& 0.8117  &  0.8411      &  0.8316  &0.8926   & 0.8248          \\ \hline
		mLB	& t& 0.7138  &   0.5141        & 0.5141  & 0.3581  &            \\
		& IV & 0.5248  & 0.3283   & 0.2145  & 0.1800  &            \\
		& PS &0.5261   & 0.3102       & -0.0800  &  -0.2382 &            \\
		& HR & 0.4994  &  0.1141     &-0.1680   &-0.2578   &            \\ \hline
		sLB	& t& 0.0925  &  0.0826       & 0.0525  & 0.0454  &            \\
		& IV  & 0.0798  & 0.0718      &  0.0517 & 0.0492  &            \\
		& PS & 2.3219  & 3.1147      &2.3654   & 1.7164  &            \\
		& HR  & 0.1601  &  0.3372         &0.2632    & 0.2038  &             \\ \hline
		\multicolumn{2}{|c|}{NS($n=400$)}		  &2000                   &  1999 &  861.33   &    40.25      &                   \\\hline
		CR	& t & 0.0838    & 0.0830          & 0.0004     &  0 &  0.0681         \\
	& IV& 0.8310    & 0.8309         &  0.6146    & 0.0637  &  0.7948         \\
	&  PS& 0.1973    &  0.1666        & 0.3383     &  0.4013 &   0.2112        \\
	& HR& 0.8103    &  0.8084        & 0.8053     &  0.8280 & 0.8087          \\ \hline
	mLB	& t & 0.6846    &0.4859           &   0.3148    &  0.2948 &           \\
	& IV &  0.5495   & 0.3509         & 0.1903     & 0.1551  &           \\
	& PS &0.5498     & 0.4305         &  -0.1472    & -0.6878  &           \\
	& HR &  0.5517    & 0.3400         &   -0.0390   &  -0.1568  &           \\ \hline
	sLB	& t&  0.0608   &  0.0602        &  0.0428    &  0.0323 &           \\
	& IV  & 0.0528    & 0.0528          &   0.0392   &   0.0327 &           \\
	& PS & 2.7299    &  1.5039        & 5.5870      & 3.8454  &           \\
	& HR & 0.0553    & 0.0959         & 0.2240     & 0.1886  &           \\ \hline
	\end{tabular}
\end{table}

\begin{table}[!h]
	\centering
	\caption{Coverage rate and estimated mean lower bound and its standard deviation based on 4 different methods under the setting of \textbf{GARCH}.}
	\begin{tabular}{|c c |c c c c c  |}
		\hline
		\multicolumn{2}{|c|}{ $\beta_j$}		  & 0.6         & 0.4                                         & 0.2  & 0.1 & Overall \\  \hline
		\multicolumn{2}{|c|}{NS($n=200$)}		  &2000  &  2000                  &  1705 &   195  &                         \\\hline
		CR	& t &  0.2180   &  0.2355         &  0.1734   &  0 & 0.2038         \\
		& IV&  0.9150   &   0.9195          &  0.9069    &0.5282   &  0.9014        \\
		&  PS& 0.3057   &  0.3230         & 0.4528    & 0.4090  & 0.3574         \\
		& HR& 0.7782   &  0.7945         &  0.8184   &   0.8141 &  0.7966        \\ \hline
		mLB	& t &   0.6351 & 0.4330          &  0.2374   & 0.1821  &          \\
		& IV &0.5387    & 0.3376          &  0.1473   & 0.0978  &          \\
		& PS & 0.5692   &0.3910            &  0.0417   & -0.3196  &          \\
		& HR &  0.5636  &    0.3542        &  0.0211   & -0.0723  &          \\ \hline
		sLB	&t   &   0.0456  &  0.0453         &  0.0397   &  0.0326 &                   \\
		& IV& 0.0454   &  0.0459         &  0.0400   &0.0326   &          \\
		& PS & 0.7576   & 0.4756           &  2.3355   &  3.4137 &          \\
		& HR & 0.0489   & 0.0801          &   0.1772   & 0.1247   &               \\ \hline
		\multicolumn{2}{|c|}{NS($n=400$)}		  &2000  &  2000                  &  1994 & 398.5   &                 \\\hline
		CR	& t &  0.2303  &  0.2255          & 0.2265    & 0.0006  &  0.2132        \\
		& IV&   0.8962  &  0.9030         &    0.8965  &  0.5169 & 0.8747         \\
		&  PS&0.2955    &   0.2760         & 0.4505    & 0.3808  & 0.3430         \\
		& HR& 0.7710   &   0.7775         & 0.7793   & 0.7246  & 0.7727         \\ \hline
		mLB	& t & 0.6219   &  0.4220         &  0.2225    & 0.1471  &          \\
		& IV & 0.5632   & 0.3633          &  0.1639   & 0.0989  &          \\
		& PS &0.5963    & 0.4090          &  0.1208   &  -0.0140 &          \\
		& HR &  0.5783  &   0.3783        &  0.1579   & -0.0023  &          \\ \hline
		sLB	& t& 0.0293   &  0.0291         & 0.0294    &0.0166   &          \\
		& IV  & 0.0292   &   0.0289         & 0.0287    &  0.0197 &          \\
		& PS & 0.2786   &  0.2240          &  1.0688   &  0.8820 &          \\
		& HR & 0.0291   &  0.0288          &  0.0812   & 0.1054  &          \\\hline
	\end{tabular}
\end{table}

\begin{table}[!h]
	\centering
	\caption{Coverage rate and estimated mean lower bound and its standard deviation based on 4 different methods under the setting of  \textbf{AR}.}
	\begin{tabular}{|c c |c c c c  c |}
		\hline
			\multicolumn{2}{|c|}{ $\beta_j$}		  & 0.6         & 0.4                                         & 0.2  & 0.1 & Overall \\  \hline
		\multicolumn{2}{|c|}{NS($n=200$)}		  &2000  &  1881                  &  476.33 &   51.75   &      \\\hline
			CR & t & 0.2495    &  0.0712        & 0     & 0  &   0.1435        \\
		& IV&  0.8822   &  0.8453        &  0.5430    & 0.1594  & 0.8213          \\
		&  PS& 0.2550    &    0.2605        &  0.3555    & 0.3430  &  0.2692         \\
		& HR&  0.8333   &  0.8522        &  0.8824    &  0.9275 &   0.8477        \\ \hline
		mLB	& t &  0.6699   &  0.5312        &  0.3875    &0.3590   &           \\       
		& IV &   0.5189  &  0.3230        &  0.1965    &  0.1503 &           \\
		& PS & 0.6598    & 0.2845         &   -0.1354   & -0.5437  &           \\
		& HR & 0.5195    &   0.1129        &  -0.2232    & -0.2850  &           \\ \hline
		sLB	&t   & 0.0562 & 0.0850          &  0.0593    & 0.0542   &         \\
		& IV  &  0.1513    & 0.0773         &   0.0598    & 0.0587  &         \\
		& PS &  0.3930   & 3.8635          &  3.0106    &  2.5480  &           \\
		& HR  &  0.1443   & 0.3332         & 0.2461     &  0.1773 &           \\ \hline
		\multicolumn{2}{|c|}{NS($n=400$)}		  &2000  &  2000                  &  1015.66 & 74.75   &                 \\\hline
		CR	& t & 0.2392    & 0.0620         &  0.0003    &0   &  0.1184           \\
		& IV&  0.8520    & 0.8375         & 0.6935     &  0.0836  &   0.8033          \\
		&  PS& 0.2513    &  0.1700        & 0.3656     & 0.3946  &  0.2442         \\
		& HR&  0.7928   &  0.8245        &  0.8336    &  0.8595 &   0.8143        \\ \hline
		mLB	& t & 0.6473    &  0.4904        &   0.3150    & 0.2895   &           \\
		& IV & 0.5670    & 0.3479          & 0.1822     &0.1457   &           \\
		& PS &  0.6449     & 0.4459          & -0.1665     & -0.4293  &           \\
		& HR & 0.5696  & 0.3345           & -0.0647     &  -0.1900  &           \\ \hline
		sLB	& t&  0.0389   &0.0633           &  0.0446    & 0.0353  &           \\
		& IV  & 0.0346    & 0.0538           &   0.0397    &   0.0344 &           \\
		& PS &  0.1670     & 0.6245         &  5.0879     & 2.6451  &           \\
		& HR & 0.0393   & 0.0988         &   0.2242   &  0.1656  &           \\ \hline
	\end{tabular}
\end{table}

\begin{table}[!h]
	\centering
	\caption{Coverage rate and estimated mean lower bound and its standard deviation based on 4 different methods under the setting of \textbf{IID}.}
	\begin{tabular}{|c c |c c c c c |}
		\hline
		\multicolumn{2}{|c|}{ $\beta_j$}		  & 0.6         & 0.4                                         & 0.2  & 0.1 & Overall \\  \hline
		\multicolumn{2}{|c|}{NS($n=200$)}		  &1999.5  &  1791                  &  57.33 &   0.75  &      \\\hline
			CR	& t & 0.7954    &  0.7884        & 0     & 0  &  0.7802         \\
		& IV& 0.7904    & 0.7923          & 0     & 0  & 0.7795          \\
		&  PS& 0.7709    &   0.7962       &   0.1570    &  0 & 0.7734          \\
		& HR&  0.7849   & 0.7839         &  0.2558    & 0  & 0.7764          \\ \hline
		mLB	& t & 0.5522     &0.3600           &  0.2717    & 0.2626  &           \\       
		& IV &  0.5524   & 0.3601         &  0.2729    &0.2643   &           \\
		& PS & 0.2916    & 0.3587         & -0.1950     &  0.2327 &           \\
		& HR & 0.5507    &  0.3352        &  0.2104    &  0.2297 &           \\ \hline
		sLB	&t   & 0.0581 & 0.0488         & 0.0232     &  0.0239   &           \\
		& IV  & 0.0586    & 0.0493         & 0.0244     & 0.0257 &           \\
		& PS &  9.4395   & 0.2178          &  2.8080     &  0.0324  &           \\
		& HR  &  0.0647   & 0.1045         &  0.1230    &  0.1281 &           \\ \hline
		\multicolumn{2}{|c|}{NS($n=400$)}		  &2000  &  2000                  &  303.33 & 0.25   &                 \\\hline
		CR	& t & 0.7953    &  0.8015        &   0.0626    &  0 & 0.7471          \\
		& IV& 0.7950    &  0.7995           &  0.0813    & 0  & 0.7467          \\
		&  PS&  0.7895   & 0.8015         &   0.4220    & 0  &  0.7691         \\
		& HR&   0.7708  &    0.7945      &  0.4868    & 0 &  0.7618         \\ \hline
		mLB	& t & 0.5664    &  0.3663        &  0.2264    & 0.2330  &           \\
		& IV & 0.5665    &  0.3665        &0.2261      &  0.2300 &           \\
		& PS &  0.5060   &   0.3681       &  0.1787    &0.1221  &           \\
		& HR &  0.5680   & 0.3649         &  0.1625    &  0.2100     &           \\ \hline
		sLB	& t& 0.0404    & 0.0392         &  0.0196    & 0  &           \\
		& IV  & 0.0404    &  0.0393         &   0.0203   & 0  &           \\
		& PS & 0.6527    & 0.0486         &    0.1853
		   &  0 &           \\
		& HR &   0.0423   &  0.0456        &   0.1195   & 0   &           \\ \hline
	\end{tabular}
\end{table}

\begin{table}[!h]
	\centering
	\caption{Coverage rate and estimated mean lower bound and its standard deviation based on 4 different methods under the setting of \textbf{MVN}.}
	\begin{tabular}{|c c |c  c c c c |}
		\hline
		\multicolumn{2}{|c|}{ $\beta_j$}		  & 0.6         & 0.4                                         & 0.2  & 0.1 & Overall \\  \hline
		\multicolumn{2}{|c|}{NS($n=200$)}		  &1994.5  &  1555                 &  97.3 &   6.5   &        \\\hline
			CR	& t & 0.1943    & 0.0695          & 0     &  0 &  0.1356         \\
		& IV&   0.8042   &  0.7646        & 0.0274     & 0  &  0.7652         \\
		&  PS&  0.3256   & 0.3164         &  0.2774    &   0.3462 &   0.3204        \\
		& HR&  0.7927   &0.8064          &   0.7021    &  0.6538 &  0.7958         \\ \hline
		mLB	& t & 0.6798    & 0.4946          &  0.4130    &  0.4136 &           \\       
		& IV & 0.5244    &  0.3497        &  0.2795    & 0.2540  &           \\
		& PS &  0.4785   &  0.2704         &  -0.1790    & -0.0480  &           \\
		& HR &  0.4908   &  0.1377        &  -0.0552    & -0.1234  &           \\ \hline
		sLB	&t   &0.0911  &  0.0692        &  0.0482     &  0.0428 &           \\
		& IV  &   0.0872   &  0.0705        &   0.0484   & 0.0417  &           \\
		& PS &  2.0655    & 2.0499           &  3.0333    &    0.5671 &           \\
		& HR  &  0.1802   &   0.3162       &  0.2839    & 0.3053  &           \\ \hline
		\multicolumn{2}{|c|}{NS($n=400$)}		  &2000  &  1980                  & 312.66 & 5.5   &                 \\\hline
		CR	& t &0.1430     & 0.1217         &   0   & 0  &     0.1226      \\
		& IV&  0.8122   &  0.8061        &  0.2228    &  0 &   0.7654        \\
		&  PS&  0.2690   & 0.2152          &  0.2409     &  0 &0.2418           \\
		& HR&  0.7993   & 0.8056          &  0.6866    &   0.4545 &  0.7929         \\ \hline
		mLB	& t &  0.6697   &  0.4712        &   0.3298   & 0.3251  &           \\
		& IV & 0.5475    &  0.3492        &  0.2273    &  0.2171 &           \\
		& PS &  0.5054   &  0.3951        &  0.1222    &  0.2969 &           \\
		& HR & 0.5491    &   0.3119       &   0.0251   &  0.0178 &           \\ \hline
		sLB	& t& 0.0634    & 0.0304 &0.0619          &  0.0281    &              \\
		& IV  &  0.0594   &  0.0577        &  0.0357    & 0.0316  &           \\
		& PS &  2.7383
		   &  1.2978        & 1.2390     &       0.2187   &    \\
		& HR &  0.0629   &  0.1526        &  0.2215
		     & 0.1798  &           \\ \hline
	\end{tabular}
\end{table}

\section{Concluding Remarks}
The problem of constructing confidence intervals after selection is discussed in Benjamini and Yekutieli (2005). They notice that "it is common practice to ignore the issue of selection and multiplicity when it comes to multiple confidence intervals, reporting a selected subset of intervals at their marginal (nominal, unadjusted) level. Confidence intervals are not corrected for multiplicity even when the only reported intervals  are those for the statistically significant parameters" and point out that ``the selection of the parameters for which confidence interval estimates are constructed or highlighted tends to cause reduced average coverage, unless their level is adjusted." They present a procedure to adjust confidence intervals so that ``the expected proportion of parameters not covered by their confidence intervals among the selected parameters, where the proportion is 0 if no parameter is selected", which they called false coverage rate (FCR), is controlled at a predetermined level $\alpha$. This idea is similar to that of controlling false discovery rate (FDR) in multiple hypothesis testing problem. Suppose there are $p$ null hypotheses $H_1,\ldots,H_p$ with the corresponding $p$-values $q_1,\ldots,q_p$ such that if we reject $H_i$ when $q_i<\alpha$, then the probability of false rejection, $\mathbf{P}(q_i<\alpha\mid H_i)$, is less than or equal to $\alpha$. Without adjusting $p$-values, the expected proportion of false rejection can be greater than $\alpha$. A popular procedure to control FDR is Benjamini-Bochberg (BH) procedure; see Benjiamini and Hochberg (1995). In BH procedure, after sorting the $p$-values $q_1\leq \ldots\leq q_p$, hypothesis $H_i$ is rejected if $p_{i}\leq \frac{m\alpha}{p}$, where $m$ is the number of rejection. Note that it would be very hard for a hypothesis to be rejected if $m\ll p$. Similarly, Benjamini and Yekutiel's method of adjusting in our problem is to construct a marginal confidence interval with confidence level $1-\frac{m}{p}\alpha$ for controlling FCR at $\alpha$ level. It is usually too wide when $m\ll p$. Although their procedure can handle selection effect, our simulations show that there is little selection effect for strong signals, and hence their method is not appropriate for our problem.

Besides, similar to the method proposed by Taylor et al.\ (2014), Benjiamini and Yekutieli's method cannot handle spill-over effect. Ing et al.\ (2017) notice such effect after OGA selection for finite $n$. They define the null hypotheses in terms of the semi-population version of OGA, and use the properties of OGA and closed testing principle to develop a procedure for controlling family-wise error rate, the probability of existing false rejection, for testing if $\beta_j=0$ for $j\in \hat{J}$. However, it is unclear how their approach can be extended for constructing confidence intervals for the selected coefficients in martingale regression model. Moreover, our approach does not require $\hat{J}$ to be selected by OGA. As long as a select method can consistently select coefficients that are significantly greater than 0 as $n\rightarrow \infty$, then the algorithms presented in Section 2 still work with OGA replaced by that select method as the theorems in Section 2 do not depend on any select methods. 

The performance of our approach depends on the appropriateness of using factor models for $\mathbf{X}$ and $\boldsymbol{\varepsilon}$ in model \eqref{Linear Model}. Our procedure can modified accordingly if we have some prior knowledge of $\mathbf{X}$ and $\boldsymbol{\varepsilon}$  to improve performance. For instance, if we know that $\boldsymbol{\varepsilon}$ follow GARCH(1,1) model as in Example 2, instead of resampling on $\hat{\boldsymbol{\varepsilon}}$ in Step 4 of Algorithm 4, we can first fit a GARCH(1,1) model on $\hat{\boldsymbol{\varepsilon}}$ to get the estimates of the parameters and $\xi_t, t=1,\ldots,n$, and then do resampling on estimated $\hat{\xi}_t$. However, even with wrong modeling, our approach can still achieve improvement over other methods that do not handle both selection effect and spill-over effect in our simulations.
\section{Appendix}
To prove the main results we need the following results. From Section APPENDIX of Bai and Ng(2002), we have 
\begin{align*}
&\norm{\mathbf{H}}=O_p(1), \norm{\hat{\mathbf{D}}^{-1}}=O_p(1), \norm{\hat{\mathbf{D}}-\mathbf{D}}=O_p(n^{-1/2}),\\ &\norm{\hat{\mathbf{D}}^{-1}-\mathbf{D}^{-1}}=O_p(n^{-1/2})
\end{align*}
Where $\hat{\mathbf{D}}=\frac{1}{n}\hat{\mathbf{F}}^T\hat{\mathbf{F}}$ and $\mathbf{D}=\frac{1}{n} \mathbf{H}^T\mathbf{F}^T\mathbf{F}\mathbf{H}$. Note that here we assume $\hat{\mathbf{F}}^{\hat{k}}=\hat{\mathbf{F}}^r=\hat{\mathbf{F}}$ by Theorem 2.

\begin{lemma}
	Let  $\mathbf{P}_{\hat{F}}=\frac{1}{n}\hat{\mathbf{F}}\hat{\mathbf{D}}^{-1}\hat{\mathbf{F}}^T$ and $\mathbf{P}_{FH}=\frac{1}{n}\mathbf{F}\mathbf{H}\mathbf{D}^{-1}\mathbf{H}^T\mathbf{F}^T=\mathbf{F}(\mathbf{F}^T\mathbf{F})^{-1}\mathbf{F}^T$. We have
	\begin{align*}
	\norm{\mathbf{P}_{\hat{F}}-\mathbf{P}_{FH}}=\norm{\mathbf{P}_{\hat{F}}^{\bot}-\mathbf{P}_{FH}^{\bot}}=O_p(\frac{1}{\sqrt{n}})
	\end{align*}
\end{lemma}

\begin{proof}[Proof of Lemma 2]
	\begin{align*}
	\norm{\mathbf{P}_{\hat{F}}-\mathbf{P}_{FH}}&=\frac{1}{n}\parallel\hat{\mathbf{F}}\hat{\mathbf{D}}^{-1}\hat{\mathbf{F}}^T-\mathbf{F}\mathbf{H}\hat{\mathbf{D}}^{-1}\hat{\mathbf{F}}^T+\mathbf{F}\mathbf{H}\hat{\mathbf{D}}^{-1}\hat{\mathbf{F}}^T-\mathbf{F}\mathbf{H}\hat{\mathbf{D}}^{-1}\mathbf{H}^T\mathbf{F}^T\\
	&+\mathbf{F}\mathbf{H}\hat{\mathbf{D}}^{-1}\mathbf{H}^T\mathbf{F}^T-\mathbf{F}\mathbf{H}\mathbf{D}^{-1}\mathbf{H}^T\mathbf{F}^T\parallel\\
	&\leq \frac{1}{n}\norm{\hat{\mathbf{F}}-\mathbf{F}\mathbf{H}}\norm{\hat{\mathbf{D}}^{-1}}\norm{\hat{\mathbf{F}}}+\frac{1}{n}\norm{\mathbf{F}\mathbf{H}}\norm{\hat{\mathbf{D}}^{-1}}\norm{\hat{\mathbf{F}}-\mathbf{F}\mathbf{H}}\\
	&+\frac{1}{n}\norm{\mathbf{F}\mathbf{H}}^2\norm{\hat{\mathbf{D}}^{-1}-\mathbf{D}^{-1}}\leq O_p(\frac{1}{\sqrt{n}})
	\end{align*}
	Note that 
	\begin{align*}
	&\sqrt{\frac{1}{n}}\norm{\hat{\mathbf{F}}-\mathbf{F}\mathbf{H}}=\sqrt{\frac{1}{n}\sum_{t=1}^{n}\norm{\hat{f}_t-\mathbf{H}^Tf_t}^2}=O_p(\frac{1}{\sqrt{n}})\\
	&\sqrt{\frac{1}{n}}\norm{\hat{\mathbf{F}}}\leq\sqrt{\frac{1}{n}}\norm{\hat{\mathbf{F}}-\mathbf{F}\mathbf{H}}+\sqrt{\frac{1}{n}}\norm{\mathbf{F}}\norm{\hat{\mathbf{H}}}=O_p(1)
	\end{align*}
\end{proof}

%

\begin{lemma}
	\begin{align*}
	&E(\frac{1}{\sqrt{n}}\sum_{j\in J^c}|\beta_j|\norm{\mathbf{e}_j})\leq M^{1+
		\frac{1}{8}}\\
	&E(\frac{1}{\sqrt{n}}\sum_{j\in J^c}|\beta_j|\norm{\mathbf{E}_J^T\mathbf{e}_j})\leq \sqrt{m}M^{1+
		\frac{1}{2}}
	\end{align*}
\end{lemma}

\begin{proof}[Proof of Lemma 3]
	\begin{align*}
	E(\frac{1}{\sqrt{n}}\sum_{j\in J^c}|\beta_j|\norm{\mathbf{e}_j})&=\frac{1}{\sqrt{n}}\sum_{j\in J^c}|\beta_j| E\left(\sqrt{\sum_{t=1}^{n}e_{tj}^2}\right)\\
	&\leq \frac{1}{\sqrt{n}} \sum_{j\in J^c}|\beta_j|\sqrt{\sum_{t=1}^{n}Ee_{tj}^2}\\
	&\leq \sum_{j\in J^c}|\beta_j| \sqrt{\frac{1}{n}\sum_{t=1}^{n}M^{\frac{1}{4}}}\\
	&\leq M^{1+\frac{1}{8}}
	\end{align*}
	
	\begin{align*}
	E(\frac{1}{\sqrt{n}}\sum_{j\in J^c}|\beta_j|\norm{\mathbf{E}_J^T\mathbf{e}_j})&=E(\frac{1}{\sqrt{n}}\sum_{i\in J^c}|\beta_i|\sqrt{\sum_{j\in J}(\sum_{t=1}^{n}e_{ti}e_{tj})^2}\\
	&\leq \sum_{i\in J^c}|\beta_i| \sqrt{m E(\frac{1}{m}\sum_{j \in J}(\frac{1}{n}e_{ti}e_{tj})^2)}\\
	&\leq \sqrt{m} M^{1+\frac{1}{2}}
	\end{align*}
\end{proof}

\begin{proof}[Proof of Theorem 3]
	
	Consider
	\begin{align*}
	\frac{\tilde{\mathbf{X}}_J^T\tilde{\mathbf{X}}_J}{n}&=\frac{1}{n}(\mathbf{P}_{\hat{F}}^{\bot}\mathbf{X}_J)^T(\mathbf{P}_{\hat{F}}^{\bot}\mathbf{X}_J)\\
	&=\frac{1}{n}(\mathbf{P}_{\hat{F}}^{\bot}\mathbf{F}\mathbf{\lambda}_J+\mathbf{P}_{\hat{F}}^{\bot}\mathbf{E}_J)^T(\mathbf{P}_{\hat{F}}^{\bot}\mathbf{F}\mathbf{\lambda}_J+\mathbf{P}_{\hat{F}}^{\bot}\mathbf{E}_J)\\
	&=\frac{1}{n}(\mathbf{P}_{\hat{F}}^{\bot}\mathbf{F}\lambda_J)^T(\mathbf{P}_{\hat{F}}^{\bot}\mathbf{F}\mathbf{\lambda}_J)+\frac{2}{n}(\mathbf{P}_{\hat{F}}^{\bot}E_J)^T(\mathbf{P}_{\hat{F}}^{\bot}\mathbf{F}\mathbf{\lambda}_J)+\frac{1}{n}(\mathbf{P}_{\hat{F}}^{\bot}\mathbf{E}_J)^T(\mathbf{P}_{\hat{F}}^{\bot}\mathbf{E}_J)
	\end{align*} 	
	
	Therefore 
	\begin{align*}
	\norm{\frac{\tilde{\mathbf{X}}_J^T\tilde{\mathbf{X}}_J}{n}-\frac{\mathbf{E}_J^T\mathbf{E}_J}{n}}\leq \frac{1}{n}\norm{\mathbf{P}_{\hat{F}}^{\bot}\mathbf{F}\mathbf{\lambda}_J}^2+\frac{2}{n}\norm{\mathbf{P}_{\hat{F}}^{\bot}\mathbf{E}_J} \norm{\mathbf{P}_{\hat{F}}^{\bot}\mathbf{F}\mathbf{\lambda}_J}+\frac{1}{n}\norm{\mathbf{P}_{\hat{F}}\mathbf{E}_J}^2
	\end{align*}
	Note that 
	\begin{align*}
	\frac{1}{\sqrt{n}}\norm{\mathbf{P}_{\hat{F}}^{\bot}\mathbf{F}\mathbf{\lambda}_J} 
	&\leq \frac{1}{\sqrt{n}}\norm{(\mathbf{P}_{\hat{F}}^{\bot}-\mathbf{P}_{FH}^{\bot})\mathbf{F}\mathbf{\lambda}_J}+\frac{1}{\sqrt{n}}\norm{\mathbf{P}_{FH}^{\bot}\mathbf{F}\mathbf{\lambda}_J}\\
	&\leq \frac{1}{\sqrt{n}}\norm{\mathbf{P}_{\hat{F}}^{\bot}-\mathbf{P}_{FH}^{\bot}}\norm{\mathbf{F}}\sqrt{\sum_{j\in J}\norm{\lambda_j}^2}\\
	&\leq O_p(\frac{1}{\sqrt{n}})\sqrt{m}\bar{\lambda}=O_p(\sqrt{m/n})\\
	\end{align*}
	
	\begin{align*}
	\frac{1}{\sqrt{n}}\norm{\mathbf{P}_{\hat{F}}\mathbf{E}_J}
	&\leq \frac{1}{\sqrt{n}}\norm{\mathbf{P}_{\hat{F}}-\mathbf{P}_{FH}}\norm{\mathbf{E}_J}+\frac{1}{\sqrt{n}}\norm{\mathbf{P}_{FH}\mathbf{E}_J}\\
	&\leq O_p(\frac{1}{\sqrt{n}})+\frac{1}{\sqrt{n}}\norm{\frac{1}{n}\mathbf{F}\mathbf{H}\mathbf{D}^{-1}\mathbf{H}^T\mathbf{F}^T\mathbf{E}_J}\\
	&\leq O_p(\frac{1}{\sqrt{n}})+\frac{1}{n}(\frac{1}{\sqrt{n}}\norm{\mathbf{F}})\norm{\mathbf{H}}^2\norm{\mathbf{D}^{-1}}\sqrt{\sum_{j\in J}\norm{\mathbf{F}^T\mathbf{e}_j}^2}\\
	&=O_p(\frac{1}{\sqrt{n}})+O_p(1)\frac{1}{n}\sqrt{\sum_{j\in J}\norm{\sum_{t=1}^{n}f_te_{t_j}}^2}\\
	&=O_p(\frac{1}{\sqrt{n}})+O_p(1)\frac{1}{\sqrt{n}}\sqrt{m} = O_p(\sqrt{m/n})\\
	\end{align*}
	
	\begin{align*}
	\frac{1}{\sqrt{n}}\norm{\mathbf{P}_{\hat{F}}^{\bot}\mathbf{E}_J}
	&=\frac{1}{\sqrt{n}}\norm{\mathbf{E}_J}-\frac{1}{\sqrt{n}}\norm{\mathbf{P}_{\hat{F}}\mathbf{E}_J}\\
	&=O_p(1)+O_p(\sqrt{m/n})
	\end{align*}
	Therefore 
	\begin{align*}
	&\norm{\frac{\tilde{\mathbf{X}}_J^T\tilde{\mathbf{X}}_J}{n}-\frac{\mathbf{E}_J^T\mathbf{E}_J}{n}}=O_p(\sqrt{m/n})\rightarrow 0 \quad\text{as}\quad n\rightarrow \infty\\
	\Rightarrow & \norm{\frac{\tilde{\mathbf{X}}_J^T\tilde{\mathbf{X}}_J}{n}-\mathbf{G}_J}\leq \norm{\frac{\tilde{\mathbf{X}}_J^T\tilde{\mathbf{X}}_J}{n}-\frac{\mathbf{E}_J^T\mathbf{E}_J}{n}}+\norm{\frac{\mathbf{E}_J^T\mathbf{E}_J}{n}-\mathbf{G}_J}\rightarrow 0
	\end{align*}
	Since $\mathbf{G}_J\succ 0$, therefore $\Big(\frac{\tilde{\mathbf{X}}_J^T\tilde{\mathbf{X}}_J}{n}\Big)^{-1}=O_p(1)$
\end{proof}

\begin{proof}
	\begin{align*}
	\tilde{\boldsymbol{\beta}}_J-\boldsymbol{\beta}_J&=(\tilde{\mathbf{X}}_J^T\tilde{\mathbf{X}}_J)^{-1}\tilde{\mathbf{X}}_J^TY-\boldsymbol{\beta}_J\\
	&=(\tilde{\mathbf{X}}_J^T\tilde{\mathbf{X}}_J)^{-1}\tilde{\mathbf{X}}_J^T(\mathbf{X}_J\boldsymbol{\beta}_J+\mathbf{X}_{J^c}\boldsymbol{\beta}_{J^c}+\boldsymbol{\varepsilon})-\boldsymbol{\beta}_J\\
	&=(\tilde{\mathbf{X}}_J^T\tilde{\mathbf{X}}_J)^{-1}\tilde{\mathbf{X}}_J^T(\mathbf{P}_{\hat{F}}^{\bot}\mathbf{X}_J\boldsymbol{\beta}_J+\mathbf{P}_{\hat{F}}X_{J}\boldsymbol{\beta}_{J}+\mathbf{X}_{J^c}\boldsymbol{\beta}_{J^c}+\boldsymbol{\varepsilon})-\boldsymbol{\beta}_J\\
	&=(\tilde{\mathbf{X}}_J^T\tilde{\mathbf{X}}_J)^{-1}\mathbf{\tilde{X}}_J^T(\mathbf{F}\mathbf{\lambda}_{J^c}\boldsymbol{\beta}_{J^c}+\mathbf{E}_{J^c}\boldsymbol{\beta}_{J^c}+\boldsymbol{\varepsilon})
	\end{align*}
	Therefore
	\begin{align*}
	\norm{\tilde{\boldsymbol{\beta}}_J-\boldsymbol{\beta}_J}&\leq \frac{1}{n}\norm{\Big(\frac{(\tilde{\mathbf{X}}_J^T\tilde{\mathbf{X}}_J)}{n}\Big)^{-1}}\Big(\norm{\tilde{\mathbf{X}}_J}\norm{\mathbf{P}_{\hat{F}}^{\bot}\mathbf{F}\boldsymbol{\lambda}_{J^c}\boldsymbol{\beta}_{J^c}}+\norm{\mathbf{X}_J^T\mathbf{P}_{\hat{F}}^{\bot}\mathbf{E}_{J^c}\boldsymbol{\beta}_{J^c}}+\norm{\mathbf{X}_J^T\mathbf{P}_{\hat{F}}^{\bot}\boldsymbol{\varepsilon}}\Big)
	\end{align*}
	Note that 
	\begin{align*}
	\frac{1}{\sqrt{n}}\norm{\mathbf{P}_{\hat{F}}^{\bot}\mathbf{F}\boldsymbol{\lambda}_{J^c}\boldsymbol{\beta}_{J^c}}&\leq \frac{1}{\sqrt{n}}\norm{\mathbf{P}_{\hat{F}}^{\bot}-\mathbf{P}_{FH}^{\bot}}\norm{\mathbf{F}}\norm{\sum_{j\in J^c}\beta_j\lambda_j}\\
	&\leq O_p(\frac{1}{\sqrt{n}})\sum_{j\in J^c}|\beta_j|\norm{\lambda_j}\\
	&\leq O_p(\frac{1}{\sqrt{n}})
	\end{align*}
	
	\begin{align*}
	\frac{1}{n}\norm{\mathbf{X}_J^T\mathbf{P}_{\hat{F}}^{\bot}\boldsymbol{\varepsilon}}&=\frac{1}{n}\norm{(\mathbf{P}_{\hat{F}}^{\bot}\mathbf{F}\lambda_J+\mathbf{P}_{\hat{F}}^{\bot}\mathbf{E}_J)^T\boldsymbol{\varepsilon}}\\
	&\leq \frac{1}{n}\norm{(\mathbf{P}_{\hat{F}}^{\bot}\mathbf{F}\mathbf{\lambda}_J)^T\boldsymbol{\varepsilon}}+\frac{1}{n}\norm{\mathbf{E}_J^T\boldsymbol{\varepsilon}}+\frac{1}{n}\norm{(\mathbf{P}_{\hat{F}}\mathbf{E}_J)^T\boldsymbol{\varepsilon}}\\
	&\leq (\frac{1}{\sqrt{n}}\norm{\mathbf{P}_{\hat{F}}^{\bot}\mathbf{F}\lambda_J})(\frac{1}{\sqrt{n}}\norm{\boldsymbol{\varepsilon}})+\frac{1}{n}\sqrt{\sum_{j \in J}(\sum_{t=1}^{n}e_{tj}\varepsilon_t)^2}+(\frac{1}{\sqrt{n}}\norm{\mathbf{P}_{\hat{F}}\mathbf{E}_J})(\frac{1}{\sqrt{n}}\norm{\boldsymbol{\varepsilon}})\\
	&\leq O_p(\sqrt{\frac{m}{n}})
	\end{align*}
	
	\begin{align*}
	\frac{1}{n}\norm{\mathbf{X}_J^T\mathbf{P}_{\hat{F}}^{\bot}\mathbf{E}_{J^c}\boldsymbol{\beta}_{J^c}}&=\frac{1}{n}\norm{(\mathbf{P}_{\hat{F}}^{\bot}\mathbf{F}\mathbf{\lambda}_J+\mathbf{P}_{\hat{F}}^{\bot}\mathbf{E}_J)^T\mathbf{E}_{J^c}\boldsymbol{\beta}_{J^c}}\\
	&\leq \frac{1}{\sqrt{n}}\norm{(\mathbf{P}_{\hat{F}}^{\bot}\mathbf{F}\lambda_J)^T\mathbf{E}_{J^c}\boldsymbol{\beta}_{J^c}}+\frac{1}{n}\norm{\mathbf{E}_J^T\mathbf{E}_{J^c}\boldsymbol{\beta}_{J^c}}+\frac{1}{n}\norm{(\mathbf{P}_{\hat{F}}\mathbf{E}_J)^T\mathbf{E}_{J^c}\boldsymbol{\beta}_{J^c}}\\ 
	&\leq \frac{1}{\sqrt{n}} \norm{\mathbf{P}_{\hat{F}}^{\bot}\mathbf{F}\mathbf{\lambda}_J}\Big(\frac{1}{\sqrt{n}}\sum_{j\in J^c}|\beta_j|\norm{\mathbf{e}_j}\Big)	+\frac{1}{n}\sum_{j\in J^c}|\beta_j|\norm{\mathbf{E}_J^T\mathbf{e}_j}\\
	&+\Big(\frac{1}{\sqrt{n}}\norm{\mathbf{P}_{\hat{F}}^{\bot}\mathbf{E}_J}\Big) \Big(\frac{1}{\sqrt{n}}\sum_{j\in J^c} |\beta_j|\norm{\mathbf{E}_J^T\mathbf{e}_j} \Big)\\
	&\leq O_p(\sqrt{\frac{m}{n}})
	\end{align*}

\end{proof}


\setcounter{equation}{0} 

\vskip 14pt
\noindent {\large\bf Supplementary Materials}

Contain
the brief description of the online supplementary materials.
\par
\vskip 14pt
\noindent {\large\bf Acknowledgements}

Write the acknowledgements here.
\par

\markboth{\hfill{\footnotesize\rm FIRSTNAME1 LASTNAME1 AND FIRSTNAME2 LASTNAME2} \hfill}
{\hfill {\footnotesize\rm FILL IN A SHORT RUNNING TITLE} \hfill}

\bibhang=1.7pc
\bibsep=2pt
\fontsize{9}{14pt plus.8pt minus .6pt}\selectfont
\renewcommand\bibname{\large \bf References}
\expandafter\ifx\csname
natexlab\endcsname\relax\def\natexlab#1{#1}\fi
\expandafter\ifx\csname url\endcsname\relax
  \def\url#1{\texttt{#1}}\fi
\expandafter\ifx\csname urlprefix\endcsname\relax\def\urlprefix{URL}\fi

\lhead[\footnotesize\thepage\fancyplain{}\leftmark]{}\rhead[]{\fancyplain{}\rightmark\footnotesize{} }

\vskip .65cm
\noindent
Ka Wai Tsang 
\\{\it The Chinese University of Hong Kong, Shenzhen}\\
{\it School of Science \& Engineering and Center for Statistical Science}
\vskip 2pt
\noindent
E-mail: (kwtsang@cuhk.edu.cn)
\vskip 2pt

\noindent
Wei Dai 
\\{\it The Chinese University of Hong Kong, Shenzhen}\\
{\it School of Science \& Engineering and Center for Statistical Science}
\vskip 2pt
\noindent
E-mail: (216019004@link.cuhk.edu.cn)
\end{document}